\documentclass[11pt, letterpaper]{article}

\usepackage{geometry}
\usepackage{color}
\usepackage{verbatim}
\usepackage[utf8x]{inputenc}
\usepackage{t1enc}
\usepackage[english]{babel}

\usepackage{amsmath,amssymb,amsthm}
\usepackage{mathrsfs,esint}
\usepackage{graphicx,wrapfig}
\usepackage[format=hang]{caption}

\newcommand*{\R}{{\mathbb R}}

\newcommand*{\N}{{\mathbb N}}
\newcommand*{\Z}{{\mathbb Z}}

\newcommand*{\eps}{\varepsilon}

\newcommand*{\Om}{\Omega}
\newcommand*{\calE}{\mathcal{E}}
\newcommand*{\calA}{\mathcal{A}}

\newcommand*{\pip}{\varphi}

\providecommand*{\vint}[1]{\mathchoice
          {\mathop{\vrule width 5pt height 3 pt depth -2.5pt
                  \kern -9pt \kern 1pt\intop}\nolimits_{\kern -5pt{#1}}}
          {\mathop{\vrule width 5pt height 3 pt depth -2.6pt
                  \kern -6pt \intop}\nolimits_{\kern -3pt{#1}}}
          {\mathop{\vrule width 5pt height 3 pt depth -2.6pt
                  \kern -6pt \intop}\nolimits_{\kern -3pt{#1}}}
          {\mathop{\vrule width 5pt height 3 pt depth -2.6pt
                  \kern -6pt \intop}\nolimits_{\kern -3pt{#1}}}}
\newcommand*{\jint}{\fint}

\DeclareMathOperator{\dist}{dist}
\DeclareMathOperator{\diam}{diam}
\DeclareMathOperator{\rad}{rad}

\DeclareMathOperator{\Mod}{Mod}

\numberwithin{equation}{section}
\theoremstyle{plain}
\newtheorem{thm}[equation]{Theorem}

\newtheorem{prop}[equation]{Proposition}

\newtheorem{lem}[equation]{Lemma}

\theoremstyle{definition}

\newtheorem{defn}[equation]{Definition}

\newtheorem{remark}[equation]{Remark}

\begin{document}

\title{Neumann problems for $p$-harmonic functions, and induced nonlocal operators in metric measure spaces}
\author{Luca Capogna, Josh Kline, Riikka Korte,\\ Nageswari Shanmugalingam, Marie Snipes
\thanks{The research of N.S.~and J.K.~are partially funded by the NSF grants~\#DMS-1800161 and~\#DMS-2054960.
The research of M.S.~is partially funded by a supplemental grant from NSF~\#DMS-1800161.
The research of L.C.~is partly funded by NSF~\#DMS-195599.  R.K.  was supported by Academy of Finland, project 308063.
The authors thank Mathav Murugan and Zhen-Qing Chen for illuminating discussion on jump processes and for
pointing out the
references~\cite{CK1, CK2, CKW} to us when N.S. and they were visiting at MSRI in Spring 2022, and Yannick Sire for
pointing out the reference~\cite{CSt}.}}
\maketitle

\begin{abstract}
Following ideas of Caffarelli and Silvestre  in~\cite{CS}, and using recent progress in hyperbolic fillings, we define
fractional $p$-Laplacians
$(-\Delta_p)^\theta$ with $0<\theta<1$ on any compact, doubling metric measure space  $(Z,d,\nu)$, and  prove existence,
regularity and stability for the non-homogenous non-local equation
$(-\Delta_p)^\theta u =f.$
These results, in turn, rest on the new
 existence, global H\"older regularity and stability theorems that we prove for the Neumann problem for $p$-Laplacians $\Delta_p$,
$1<p<\infty$, in bounded domains of  measure metric spaces endowed with a doubling measure that supports a Poincar\'e
inequality.  Our work also 
includes as special cases much of the previous results
by other authors in the Euclidean,
Riemannian and Carnot group settings. Unlike other recent contributions in the metric measure spaces context, our work
does not rely on the assumption that  $(Z,d,\nu)$  supports a Poincar\'e inequality.
\end{abstract}
\tableofcontents
\section{Introduction}

The goal of this paper is to
construct and study a notion of the fractional $p$-Laplacian in the setting of compact doubling metric measure spaces, by extending
Caffarelli and Silvestre's approach in~\cite{CS}. 
Our strategy consists of two main steps: First, we continue
the study begun in~\cite{MS}, of the Neumann boundary value problem for the
$p$-Laplacian, $1<p<\infty$,
\begin{align*}
\Delta_p u:=\text{div}(|\nabla u|^{p-2}\nabla u)=0 &\text{ in }\Om,\\
|\nabla u|^{p-2}\partial_\eta u=f &\text{ on }\partial\Om,
\end{align*}
expressed in the weak form
\[
\int_\Om|\nabla u|^{p-2}\nabla u\cdot\nabla\phi\, d\mu=\int_{\partial\Om}\phi f\, d\nu
 \quad \text{ for every Sobolev function } \phi,\]
where $\Om$ is an open domain in a metric measure space $(X,d,\mu)$ equipped with a doubling measure supporting a Poincar\'e
inequality. In this general setting  $\nabla u$ indicates either the gradient in a Cheeger differential structure (see ~\cite{Che}) or
in an infinitesimally Hilbertian structure  in the sense of Gigli~\cite{AGS1, AGS2,G, GM},  and hence provides
a unified framework for a wide class of operators.  We prove global H\"older
regularity and stability results with respect to the Neumann data in suitable Lebesgue classes.

Second,  using our new results for the Neumann problem, we provide a construction of a family of fractional
$p$-Laplacian operators on the boundary of the domain.
In fact, thanks to Theorem~\ref{thm:equiv-intro},
we can construct an analog of the fractional Laplacian for
{\it any} compact doubling metric measure space $(Z,d_Z,\nu)$ even if it does not support a $2$-Poincar\'e inequality.
We do this 
by constructing a locally compact but
non-compact doubling metric measure space $(X,d,\mu)$ that supports a $1$-Poincar\'e inequality such that
$Z$ is biLipschitz equivalent to $\partial X=\overline{X}\setminus X$, and considering solutions to the corresponding Neumann
boundary value problem in $X$. Such a construction of a metric measure space $X$ as a uniformization of a hyperbolic filling of $Z$
can be found in~\cite{BBS}. We will discuss this in detail in Section~\ref{sub:hyp-fill} below.
As a consequence, we prove
existence, regularity, Harnack inequality and stability for  solutions of the associated non-homogenous equation
 \[
(-\Delta_p)^\theta v=f \text{ in }Z,
\]
 with $\theta\in (0,1)$. We  also prove that for $p=2$, the fractional operators defined here are the same as those appearing
 in the literature,  so that our results extend  (and occasionally sharpen) earlier work by other authors in the setting of Euclidean
 spaces \cite{CS, ST}, Riemannian manifolds \cite{BGS},  Carnot groups \cite{FF}, and even in metric measure spaces \cite{EbGKSS}.

We remark explicitly that even in the special case $p=2$, part of the novelty of our approach is that it
allows one to define and study fractional Laplacians in metric measure spaces that
support neither a Poincar\'e inequality nor a corresponding regular Dirichlet form, such as the
Von Koch snowflake $K$, or the Rickman rug $[0,1]\times K$.

Next, we proceed by outlining the relevant assumptions of the paper, before stating our main results.

\noindent{\bf Structure  hypotheses:}
Throughout the paper we let $1<p<\infty$, and $\Omega$ be an  open, connected domain in
a complete metric measure space $(X,d,\mu)$ such that:
\begin{enumerate}
\item[(H0)] $\Omega$ is a John domain as defined in Subsection~\ref{subsect:John}.
\item[(H1)] $(\overline \Om, d, \mu\vert_{\overline \Om})$ is doubling and supports a $p$-Poincar\'e inequality
as in Subsection~\ref{subsect:doubling} and Subsection~\ref{subsect:2.2}.
\item[(H2)]
The boundary $\partial \Omega$ is equipped with a Radon measure $\nu$ for which there are constants $C\ge 1$
and $0<\Theta<p$ such that for all $x\in Z$
and $0<r<2\diam(\partial \Omega)$,
\begin{equation}\label{eq:Co-Dim-intro}
 \frac{1}{C}\, \frac{\mu(B(x,r)\cap\Om)}{r^\Theta}\le \nu(B(x,r))\le C\, \frac{\mu(B(x,r)\cap\Om)}{r^\Theta};
\end{equation}
that is, $\nu$ is a $\Theta$-codimensional Hausdorff measure with respect to $\mu\lfloor_\Om$.
\end{enumerate}

\begin{remark}
The constants associated with $p$, the John domain condition, and
the above two conditions will be referred to as the structural constants.

Recall that John domains are necessarily bounded domains.
 Condition~(H0) can be waived
for Theorem~\ref{thm:equiv-intro} and Theorem~\ref{thm:holder-intro}.
The boundedness of $\Om$ is used by us only to know
that solutions to the Neumann boundary value problem exist (see~\cite{MS}), and
Condition~(H0) is used only to know the existence of traces of Sobolev functions on
$\Om$ and ensure that the trace lies in the suitable Besov class of functions on the boundary of $\Om$.
For unbounded domains, John condition
 can be replaced with the assumption that $\Om$ is a uniform domain in order to obtain local trace estimates
 as in~\cite{M}, and this is sufficient for Theorem~\ref{thm:equiv-intro} and Theorem~\ref{thm:holder-intro} in the
 event that $\Om$ is not bounded. Note that~\cite{CS, FF} establish neither existence nor stability results for the
 fractional Laplacian, but focus only on local regularity such as H\"older continuity and Harnack inequality, under the
 assumption that the solution to the fractional Laplacian problem exists. Hence our results recover those of~\cite{CS, FF}
 in the case that $\partial\Om$ is a Euclidean space or a Carnot group.
\end{remark}

\begin{remark}
Our choice of $0<\Theta<p$ in~(H2) corresponds to the choice of $a$ in~\cite{CS}. In~\cite{CS}, for each $0<\theta<1$
the choice $-1<a=1-2\theta<1$ is made in considering the weight $y^a$ imposed on the domain $\Om=Z\times(0,\infty)$,
and in this paper, we need $\Theta=2(1-\theta)$ (with $p=2$) correspondingly.
\end{remark}

\begin{remark}
The structure hypotheses are more general than those assumed in the literature so far.
In fact,  domains with fractal boundaries such as in the
vonKoch snowflake domain satisfy these conditions for some $\Theta\ne 1$.
Moreover, such a flexibility allows us to consider all powers $\theta\in (0,1)$ rather than just $\theta=1/2$.
These structural assumptions are satisfied
by the contexts studied in~\cite{CS, FF,ST}, where $\Om=X\times(0,\infty)$ with $X$ either a Euclidean space or a Carnot
group, and $\Om$ equipped with a weighted product measure.
Moreover, \emph{every} compact doubling metric measure space $X$ is the boundary of a John domain (and in
fact, the boundary of a uniform domain) satisfying our structural conditions, see~\cite{BBS}.
\end{remark}

The first main theorem of this paper, stated next, is the main tool we use to construct the fractional operator on the metric space
$(\partial\Om,d,\nu)$ by considering analogs of the Dirichlet-to-Neumann transformation. Condition~(c) of this theorem tells us
that the trace of the solution $u$ on $\partial\Om$ should belong to the domain of the fractional operator.

\begin{thm}\label{thm:equiv-intro}
Suppose that conditions~(H0),~(H1) and~(H2) hold for the metric measure space $(X,d,\mu)$ and $\Om$.
Fix $1<p<\infty$, and suppose that
$f:\partial\Omega\to\R$ be in the class $L^{p'}(\partial\Om)$ where $p'$ is the H\"older
conjugate of $p$ with $\int_{\partial\Omega}f\,d\nu =0$.
Let $u\in N^{1,p}(\Omega)=N^{1,p}(\overline{\Omega})$. Then
the following are equivalent.
\begin{enumerate}
\item[\rm{(a)}] $u$ is a solution to the Neumann boundary value problem with data
  $f$ in the domain $\Omega\subset X$; that is, for all $\phi\in N^{1,p}(\overline{\Omega})$,
      \[
      \int_{\Omega}|\nabla u|^{p-2}\nabla u\cdot\nabla\phi\,d\mu=\int_{\partial\Omega}\phi f\,d\nu
      \]
  \item[\rm{(b)}] $u$ minimizes the energy functional
        \[
        I(v):=\int_{\overline{\Omega}}|\nabla v|^{p}\,d\mu-p\int_{\partial{\Omega}}v f\,d\nu
        \]
        among all functions $v\in N^{1,p}(\overline{\Omega})$.  Here we extend $\nu$ to a measure
        on the closure $\overline{\Omega}$ by zero outside of $\partial\Omega$.
  \item[\rm{(c)}] $u$ is $p$-harmonic in $\Omega$ and
      \[
      |\nabla u|^{p-2}\nabla u\cdot\nabla\eta_\epsilon\,d\mu\rightharpoonup -f\,d\nu.
      \]
      Here the convergence is that of weak convergence of signed Radon measures on $\overline{\Om}$,
      and the function $\eta_\epsilon:\overline{\Omega}\to\R$ is given\footnote{From our structural assumption~(H2), we have that
$\int_{B(\zeta,r)}|\nabla \eta_\eps|\, d\mu$ is at most a constant multiple of $\eps^{-1+\Theta}\nu(B(\zeta,2r))$ whenever
$\zeta\in\partial\Om$ and $r>0$.} by
$\eta_\eps(x)=\min\{1, \textrm{dist}(x, \overline{\Om}\setminus\Omega)/\eps\}$, and again we extend $\nu$ to a measure
      on the closure $\overline{\Omega}$ by zero outside of $\partial\Omega$.
\end{enumerate}
\end{thm}

The proof of Theorem \ref{thm:equiv-intro} is in Section~\ref{Sect:equiv-form}.

Next we turn our attention to global properties of solutions of the Neumann boundary value problem,
namely we obtain H\"older regularity up to the boundary, and stability with respect to the data. Here, the index $Q$ is the lower mass
bound exponent associated with the measure $\mu$ as in~\eqref{eq:lower-mass-exp}. Notice that we can increase the
value of $Q$ in~\eqref{eq:lower-mass-exp}, and so, if we are willing to pay the price of changing the estimates in the following
theorem, the restriction $p\le Q$ should not be considered to be a restrictive one.

\begin{thm}\label{thm:holder-intro}
Assume that~(H0),~ (H1) and (H2) hold, $1<p\le Q$, and let $u$ be a solution of the
Neumann problem as in Theorem~\ref{thm:equiv-intro},  for the boundary data $f\in L^q(\partial \Omega, \nu)$, with
$\int_{\partial\Om}f\, d\nu=0$.  If
\[
\tfrac{Q-\Theta}{p-\Theta}<q,
\]
then $u$ is $(1-\eps)$-H\"{o}lder continuous in $\overline \Omega$ with
\[
\eps=\max\left\{1-\alpha,\tfrac{\Theta(q-1)+Q}{pq}\right\},
\]
and where $\alpha$ is as in Proposition \ref{prop:KS}. Furthermore, there is a constant
$C>1$ such that if $u\ge 0$ on $\Om$ and
$W\subset\partial\Om$ is a non-empty relatively open subset of $\partial\Om$ with $f=0$ on $W$, then
whenever $x\in\Om\cup W$ and $r>0$ such that $B(x,2r)\cap\overline{\Om}\subset W\cup\Om$, we have
the Harnack inequality
\[
 \sup_{B(x,r)\cap\overline{\Om}}\, u\le C\, \inf_{B(x,r)\cap\overline{\Om}}\, u.
\]
\end{thm}

This theorem will be proved in Section~\ref{Sect:boundary-reg}, but an explanation regarding the proof is warranted here.
Even in the Euclidean setting, the perspective of metric spaces gives a new viewpoint of the Neumann problem; we can see the domain $\Omega$ as an open subset of an ambient metric measure space while also viewing
$\overline{\Omega}$ as a metric measure space in its own right. This allows us to see the solutions for Neumann boundary
value problem also as solutions to the inhomogeneous problem $-\Delta_p u=\nu_f$ on the metric measure space
$\overline{\Omega}$, where the measure $\nu_f$ is a singular measure, supported on the boundary
$\partial\Omega=\overline{\Omega}\setminus\Omega$, and is associated with the Neumann data of the original Neumann
boundary value problem. This point of view allows us to adapt a Morrey-Campanato type argument on the whole ``open set''
that is $\overline{\Omega}$.
To do so, we took inspiration from~\cite{RZ}, but as we rely on the results
from~\cite{MS} and~\cite{KS}, our proof is more direct.
The idea is to prove a version of the Morrey inequality for the metric measure
space setting. Note that with our assumption on $\partial\Omega$, we have that $\mu(\partial\Omega)=0$.
For $p=2$, the argument in the proof of H\"older continuity on regions of $\partial\Om$ where $f=0$, in the
Euclidean, Riemann manifold setting, and
Carnot setting
as established by~\cite{BGS, CS, CG, FF, ST} is based on a Harnack inequality, from which the H\"older
continuity  follows. Here  we give a more direct proof of the theorem, and in doing so, we obtain H\"older continuity even in regions of
the boundary where the Neumann data $f$ does not vanish.  H\"older regularity results  for the
non-homogeneous equation were obtained
for bounded Euclidean domains in~\cite{CSt}, and it is interesting to note that the limitations placed on the data $f$ in~\cite{CSt}
is also the limitation in the coarser setting of nonsmooth metric spaces. We also point out that, unlike the above-mentioned papers,
our discussion also includes the nonlinear setting $p\neq 2$ of the (fractional) $p$-Laplacian, and that the
Harnack inequality for the homogenous fractional PDE   follows by virtue of the global H\"older regularity of the
Neumann problem and by the results in \cite{KS}.

In terms of continuity with respect to boundary data, we prove the following,

\begin{thm}\label{thm:stability-cheeger-intro}
Let $1<p$  and $p'=\frac{p}{p-1}$, and suppose that (H0), (H1) and (H2) hold. There exists a constant
$C>0$ depending  on the structural constants in (H0), (H1), and (H2)
such that the following holds.  For boundary data $f,g\in L^{p'}(\partial \Omega, \nu)$ with
\[
\int_{\partial\Om}f\, d\nu=0=\int_{\partial\Om} g\, d\nu,
\]
denote by $u, v\in N^{1,p}(\bar \Omega)$ respectively the solutions to the
corresponding $p$-Neumann problems such that
\[
\int_\Om u\, d\mu=0=\int_\Om v\, d\mu.
\]
Then, when $p\ge 2$ we have
\[
\| \nabla u- \nabla v \|_{L^{p} (\Omega)} \le
C\, \left(\Vert f\Vert_{L^{p'}(\partial\Om)}+\Vert g\Vert_{L^{p'}(\partial\Om)}\right)^{\tfrac{1}{p(p-1)}}
\|f-g\|_{L^{p'}(\partial \Omega)}^{\tfrac1p},
\]
and when $1<p<2$ we have
\[
\| \nabla u- \nabla v \|_{L^{p} (\Omega)}
\le \left(\Vert f\Vert_{L^{p'}(\partial\Om)}+\Vert g\Vert_{L^{p'}(\partial\Om)}\right)^{\tfrac{3-p}{2(p-1)}}
\|f-g\|_{L^{p'}(\partial \Omega)}^{\tfrac12}.
\]
\end{thm}

As a consequence of
this theorem together with Lemma~\ref{lem:control-via-data} and the Poincar\'e inequality,
we see that if the boundary data converges in $L^{p'}$ to a function $f\in L^{p'}$, then the solutions converge
to a solution of the Neumann boundary value problem with boundary data $f$.
We will prove this theorem in Section~\ref{Sect-stability} below.
In addition, in Theorem~\ref{thm:stability-uppergradient} we also prove a weaker form of stability that
applies to the Neumann problem with respect to the upper gradient formulation. Theorem~\ref{thm:stability-uppergradient} is
independent of the above theorem as it considers a variant Neumann boundary value problem that arises from an energy minimization
principle that may not correspond to an Euler-Lagrange equation unless the metric measure space is
infinitesimally Hilbertian in the sense of Gigli~\cite{AGS1, AGS2, G, GM}. Such an Euler-Lagrange equation is essential for
our proof of Theorem~\ref{thm:stability-cheeger-intro}. In the setting of hyperbolic filling
as in~\cite{BBS} and in Section~\ref{sec:construct-fractLap}, there is a
Cheeger differential structure that is available to us, and so from the point of view of studying nonlocal minimization problems,
we do not lose much by considering the Cheeger differential formulation.

We are now ready to introduce and discuss the fractional (non-local) operators we are interested in.

\begin{defn}\label{def:fract-Lap-construct-intro} Let $(Z,d,\nu)$ be a metric measure space.
For any $1<p<\infty$ and $0<\theta<1$,
consider a form $\calE:L^p(Z)\times L^p(Z)\to[-\infty,\infty]$,  that is
linear in the second component and with $\calE(\alpha u,\beta v)=|\alpha|^{p-2}\alpha\beta\,\calE(u,v)$,
such that $\calE(u,u)\approx \Vert u\Vert_{B^\theta_{p,p}(Z)}$ whenever
$u$ is in the Besov class $B^\theta_{p,p}(Z)$.
We say that a function $u\in B^\theta_{p,p}(Z)$ is in the domain of the fractional $p$-Laplacian operator
$(-\Delta_p)^{\theta}$ if there is a function $f\in L^{p'}(Z)$ such that the integral identity
\[
\calE(u,\pip)=\int_Z\pip\, f\, d\nu
\]
holds for every  $\pip\in B^\theta_{p,p}(Z)$. We then denote
\[
(-\Delta_p)^{\theta} u=f\in L^{p'}(Z).
\]
\end{defn}

The above definition is concomitant with the notion of Laplacian in the theory of Dirichlet forms, see for example
the comprehensive book~\cite{FOT}.
The fractional operators we consider in this paper are associated to the form \eqref{eq:ET-go-home} below. Its
definition is based on a process reminiscent of the hyperbolic filling technique:    We will show in
Section~\ref{sec:construct-fractLap} that, given $1<p<\infty$ and $0<\theta<1$,
every compact doubling metric measure space $Z$ arises as the boundary of a uniform domain $\Om$ that is equipped
with a measure $\mu$ so that the metric measure space $X=\overline{\Om}=\Om\cup Z$, together with
$Z=\partial\Om$, satisfies conditions~(H0),~(H1) and~(H2), with  $\theta=1-\Theta/p$.
We fix a Cheeger differential structure $\nabla$ on $\Om$.
For each $u\in B^\theta_{p,p}(Z)$ we consider $\widehat{u}$ to be the unique function in $N^{1,p}(\Om)$ such that
$\widehat{u}$ is Cheeger $p$-harmonic in $\Om$ and has trace $Tr(\widehat{u})=u$ $\nu$-almost everywhere on $Z$.
We will show that $\int_\Om|\nabla \widehat{u}|^p\, d\mu\approx \Vert u\Vert_{B^\theta_{p,p}(Z)}$ (see
Section~\ref{sec:construct-fractLap} below). We then set
\begin{equation}\label{eq:ET-go-home}
\calE_T(u,v):=\int_\Om|\nabla\widehat{u}|^{p-2}\nabla\widehat{u}\cdot\nabla \widehat{v}\, d\mu.
\end{equation}
This  construction gives us a way of analyzing a wide array of fractional operators $(-\Delta_p)^{\theta}$ on
the Besov class $B^\theta_{p,p}(Z)$, for we have a broad choice of $\Om$, and for each choice of $\Om$ we then have
the flexibility of choosing a desired Cheeger differential structure as outlined in~\cite{Che} (see~\cite{HKST} for further exposition
on Cheeger differential structure). We may, instead of a Cheeger differential structure, consider the $\Gamma$-limit
of discrete differential structures as discussed in~\cite{DS}.
Should $\Om$ be infinitesimally Hilbertian in the sense of~\cite{G, GM},
we may use the differential structure associated with the infinitesimal Hilbertianity. Additional choices of structures are
available in the Euclidean setting, and a related family of non-local fractional operators were studied
by Caffarelli and Soria-Carro in~\cite{CC} using lower dimensional slices.

 Another approach to the notion of Laplacian is from the theory of Dirichlet forms, see for example~\cite{FOT}.
To emphasize this connection we introduce a new
 form $\calE_p(u,v)$,  given by
\[
\calE_p(u,v)=\int_Z\int_Z\frac{|u(y)-u(x)|^{p-2}(u(y)-u(x))(v(y)-v(x))}{\nu(B(x,d(x,y))\, d(x,y)^{\theta p}}\, d\nu(y)\, d\nu(x),
\]
with $u,v\in L^p(Z)$.
Note that $\calE_2$ is a Dirichlet form in the sense of~\cite{FOT}, and is associated with a jump process, see
also~\cite{BBCK, CK1, CK2, CKW}.
In the following theorem, $Q_Z$ is the lower mass bound exponent corresponding to $Q$ from~\eqref{eq:lower-mass-exp} for the
doubling measure $\nu$ on $Z$, and  $p'=p/(p-1)$ is the H\"older conjugate of $p$.

\begin{thm}\label{thm:main-fract-Lap-intro}
Let $(Z,d,\nu)$ be a compact doubling metric measure space, $1<p<\infty$, and $0<\theta<1$. Then
the
form $\calE_T$ on $B^\theta_{p,p}(Z)$
given by~\eqref{eq:ET-go-home} satisfies
$\calE_T(u,u)\approx\calE_p(u,u)$ for each $u\in B^\theta_{p,p}(Z)$
with the comparison constant depending solely on the doubling constant of $\nu$ and the indices $p,\theta$.  Denote by
$(-\Delta_p)^\theta$ the fractional $p$-Laplacian associated to the form $\calE_T(u,u)$. Moreover,
\begin{enumerate}
\item[{\rm (i)}] For each $f\in L^{p'}(Z)$ with $\int_Zf\, d\nu=0$ there exists a function $u_f\in B^\theta_{p,p}(Z)$ such that
$(-\Delta_p)^\theta u_f=f$ on $Z$. If $\widetilde{u_f}$ is any other such function,
then $u_f-\widetilde{u_f}$ is constant $\nu$-a.e.~in $Z$. If in addition $f\in L^q(Z)$ for some
$q>\max\{1, Q_Z/\theta\}$, then $u_f$ is H\"older continuous on $Z$.

\item[{\rm (ii)}] There exists a constant $C>0$, depending only on the structure constants, such that if $f_1,f_2\in L^{p'}(Z)$ with
$\int_{Z}f_1\, d\nu=0=\int_{Z}f_2\, d\nu$
and $u_{f_1}$, $u_{f_2}$ are the functions in $B^\theta_{p,p}(Z)$ corresponding to $f_1$, $f_2$ as above, then
\[
\Vert u_{f_1}-u_{f_2}\Vert_{L^p(Z)}\le C\, \max\{\Vert f_1\Vert_{L^{p'}(Z)},\Vert f_2\Vert_{L^{p'}(Z)}\}^\kappa\,
\Vert f_1-f_2\Vert_{L^{p'}(Z)}^\tau,
\]
with $\kappa=1/(p(p-1))$, $\tau=1/p$ when $p\ge 2$ and $\kappa=(3-p)/(2(p-1))$, $\tau=1/2$ when $1<p<2$.
\item[{\rm (iii)}] Let $W\subset Z$ be an open (nonempty) subset such that $f=0$ on $W$.
There exists a constant $C>0$ depending only on the 
structure constants, such that if $u\ge 0$ is a  solution of $(-\Delta_p)^\theta u=f$ in $Z$, then $u$ satisfies the Harnack inequality
$ \sup_B u \le C \inf_B u$ for all balls $B=B_R$ such that  $B_{4R}\subset W$.
\end{enumerate}
\end{thm}

This theorem will be proved in Section~\ref{sec:construct-fractLap}.

\begin{remark} We can weaken the hypotheses of the above theorem by replacing the compactness requirement of $Z$
with the condition that $Z$ is biLipschitz equivalent to the boundary of a uniform domain that satisfies
the conditions~(H0), (H1), and~(H2). Thus the above theorem also includes the cases when $Z$ is the entire Euclidean space
as in~\cite{CS} or even a  Carnot groups as in~\cite{FF}.
\end{remark}

In~\cite{EbGKSS} an alternate construction, based on spectral theory, for a fractional Laplacian $(-\Delta_2)^\theta$ corresponding to a Cheeger differential structure
on a complete doubling metric measure space  $(Z,d_Z,\nu)$
supporting a $2$-Poincar\'e inequality was constructed and studied. The methods used there depended strongly on
the availability of the Poincar\'e inequality on the metric space itself, and moreover, it is not adaptable to fractional powers of
the nonlinear $p$-Laplacian operator $(-\Delta_p)^\theta$.
In~\cite{EbGKSS}, the  metric space  $(Z,d_Z,\nu)$ is naturally seen as the boundary of the unbounded domain
$Z\times(0,\infty)$, where $X=Z\times[0,\infty)$ is equipped with the $\ell_2$-product metric and the product measure.
In contrast, in our present study, we consider fractional Laplacians on the boundary of a \emph{bounded} domain.
Therefore, to show that the two approaches
lead to the same notion of nonlocal Laplacian, we show that we can conformally transform $X$
into a bounded doubling metric measure space $\Om$
so that it supports a $2$-Poincar\'e inequality and that functions that are $2$-harmonic in $Z\times(0,\infty)=X$ are also
$2$-harmonic in this modified space, and  finally that $Z$ is isometric to the boundary of this modified space.

\begin{thm}\label{reconciliation}
Let $(Z,d,\nu)$ be a compact
doubling metric measure space supporting a $2$-Poincar\'e inequality, and let
$X=Z\times(0,\infty)$. There is a conformal transformation of $X$ to a metric space $\Om$ with transformed
metric $d_\rho$, together with a natural tranformation $\nu_\omega$ of the measure $\nu$, so that $(\Om,d_\rho,\nu_\omega)$
is a John domain in its completion $\overline{\Om}$ and so that $(\overline{\Om},d_\rho,\nu_\omega)$
satisfies our hypotheses~(H0),~(H1) and~(H2). Moreover, $Z$ is isometric to $\partial\Om$, and
the fractional Laplacian $(-\Delta_2)^\theta$ as constructed in Theorem~\ref{thm:main-fract-Lap-intro} above agrees with the
construction given in~\cite{EbGKSS}.
\end{thm}

This theorem will be proved in Section~\ref{Sec:Reconcile}, see Subsection~\ref{SubSec-reconciliation}.

\begin{remark}
 The fractional Neumann boundary value problem considered in~\cite[(1.6)]{CSt} for bounded Euclidean domains $U$
correspond to the solutions constructed in~\cite{EbGKSS} for the domain $\Om=\overline{U}\times(0,\infty)$. Theorem \ref{reconciliation}
 shows then  that the problem studied in~\cite[(1.6)]{CSt}
corresponds to the problem studied in the present paper, in the special setting of $p=2$ and in the Euclidean setting. We point out that the H\"older regularity result~\cite[Theorem~1.2]{CSt} requires
the same integrability condition on the  nonhomogenous data $f$ as that of Theorem~\ref{thm:holder-intro}.
\end{remark}

In the smooth setting of Euclidean spaces, Carnot groups and asymptotically hyperbolic Riemannian manifolds, there is now a vast literature on
fractional powers of the Laplacian operator, and~\cite{BGS, CS, CC, CSt, CG, ST} is merely a small sampling of the current
literature; these all are associated with the fractional powers of the \emph{linear} operator (i.e., $p=2$). The literature
in more general non-smooth setting is more limited; we refer interested readers to~\cite{BG, BGMN, BLS, CKW, EbGKSS, FF, G} as
well as the references listed therein for the linear setting.
For symmetry results related to some other
problems involving fractional $p$-Laplacian in the setting of Euclidean and Heisenberg group
settings and for the use of the method of moving planes in those settings, we refer the interested reader
to~\cite{BGJ, CL, DW} and the references listed therein.
 Analysis of non-local (linear, that is, $p=2$)
operators was undertaken in the papers~\cite{CK1, CK2, CKW} from the point of view
of jump processes and more general $\alpha$-stable processes (recall that for us, $\alpha=2\theta$) using the
language of Dirichlet forms.
 In~\cite{CK1}, Chen and Kumagai consider such processes on Ahlfors regular complete
metric measure spaces, and prove H\"older regularity, see for example~\cite[Theorem~4.14]{CK1}. In~\cite{CK2}
they extend this study to doubling metric measure spaces where a notion of uniform doubling property is also assumed;
this additional property is removed in the recent paper~\cite{CKW}. Our approach is more aligned with the approach of~\cite{CS},
as that approach is adaptable also to the nonlinear ($p\ne 2$) setting as well.

\section{Background}

In this section we gather together the needed background used in the paper.
The triple $(X,d,\mu)$ denotes a complete metric measure space with $\mu$ a Radon measure. We now list some basic
notions associated with the theory of analysis on metric measure spaces. For $x\in X$ and $r>0$, the ball centered at $x$ and with
radius $r$ is denoted $B(x,r)$, that is, $B(x,r)$ consists of all the points $y\in X$ for which $d(x,y)<r$. The closed ball
$\overline{B}(x,r)$ consists of all the points $y\in X$ for which $d(x,y)\le r$. Note that in general $\overline{B}(x,r)$ could be a
larger set than the topological closure of the open ball $B(x,r)$, but if $X$ is a length space these two sets are the same.

\subsection{ John and uniform domains} \label{subsect:John}

Recall that we assume $X$ to be a complete metric space.
We say that a domain $\Om\subset X$ is a {\it John domain} if there is a point $x_0\in\Om$, called a John center,
and a John constant $C_J\ge 1$ such that whenever $x\in\Om$, there is a rectifiable curve $\gamma_x$
in $\Om$ with end points $x_0$ and $x$ such that for each point $z$ in the trajectory of $\gamma_x$ we have
that
\[
 \dist_{X\setminus\Om}(z)\ge C_J^{-1} \ell(\gamma_x[x,z]),
\]
where $\gamma_x[x,z]$ denotes the segments of $\gamma_x$ with end points $z$ and $x$. Clearly
a John domain is a connected open set, and moreover, if $\Om\ne X$ then $\Om$ is bounded.  In this paper we also
refer to a narrower class of domains, called {\it uniform domains}, characterized by the existence of a constant
$C_U\ge 1$ such that for every pair $x,y\in \Om$ there exists a rectifiable curve $\gamma_{xy}$ joining them, with the property
\[
 \dist_{X\setminus\Om}(z)\ge C_U^{-1}  \min\bigg(\ell(\gamma_{xy} [x,z]),
 \ell(\gamma_{xy}[z,y])\bigg) \ \text{ and } \ \ell(\gamma_{xy}) \le C_U d(x,y),
\]
for all $z\in \gamma_{xy}$.  Clearly a bounded uniform domain is also John, but the converse is false.

\subsection{ Measures} \label{subsect:doubling}
We say that the measure $\mu$ is \emph{doubling} if there is a constant $C_d\ge 1$ such that
\[
0<\mu(B(x,2r))\le C_d\, \mu(B(x,r))<\infty
\]
for each $x\in X$ and $r>0$.
Doubling measures satisfy the following lower mass bound property: there are constants $c>0$
and $Q>0$ such that
for each $x\in X$, $0<r<R$, and for each $y\in B(x,R)$,
\begin{equation}\label{eq:lower-mass-exp}
c\left(\frac{r}{R}\right)^Q\le \frac{\mu(B(y,r))}{\mu(B(x,R))},
\end{equation}
see for example~\cite[page~76]{HKST}.
The constant $c$ depends solely on the doubling constant $C_d$.

The measure $\mu$ is \emph{Ahlfors} $Q$-regular for some $Q>0$ if there is a constant
$C\ge 1$ such that for all $x\in X$ and $r>0$ we have $C^{-1}r^Q\le \mu(B(x,r))\le C\, r^Q$.
It is now well-known that parts of harmonic analysis can be conducted on doubling metric measure
spaces, as described in~\cite[page~8]{Stein}, though currently there is significant headway in extending the theory of singular integrals
beyond doubling spaces. However, much of the theory of quasiconformal maps seems to require $\mu$ to be Ahlfors regular,
see for instance~\cite{HK}.
If $\mu$ is Ahlfors $Q$-regular, then $\mu$ is comparable to the $Q$-dimensional Hausdorff measure on $X$.

\subsection{Newton-Sobolev functions and Poincar\'e inequalities}\label{subsect:2.2}
We are interested in using a first-order calculus in metric measure spaces; such a first order calculus was first developed by
Heinonen and Koskela in their seminal paper~\cite{HK} in the process of investigating quasiconformal mappings between
Ahlfors regular metric spaces; see also~\cite{BBbook, HajK, HKST, Sh}.
The idea here is that one needs only the information encoded in magnitude of the gradient of a function in order to conduct
much of first-order calculus. Given a measurable function $u:X\to\R$, we say that a non-negative Borel measurable function $g$ on
$X$ is an \emph{upper gradient} of $u$ if
\[
 |u(x)-u(y)|\le \int_\gamma g\, ds
\]
for every non-constant compact rectifiable curve $\gamma$ in $X$; here, $x$ and $y$ denote the terminal points of $\gamma$.
The function $u$ is said to be in the Dirichlet--Sobolev class $D^{1,p}(X)$ (also known as the homogeneous Sobolev class)
if $u$ has an upper gradient that belongs to $L^p(X)$;
and $u$ is said to be in the Newton-Sobolev class $N^{1,p}(X)$ if it is in $D^{1,p}(X)$ and in addition, $u$ itself belongs to
$L^p(X)$. Given that upper gradients are not unique, we set the energy semi-norm on $D^{1,p}(X)$ by
\[
\mathcal{E}_p(u)^p:=\inf_g\int_Xg^p\, d\mu,
\]
where the infimum is over all upper gradients $g$ of $u$. The norm on $N^{1,p}(X)$ is given by
\[
\Vert u\Vert_{N^{1,p}(X)}:=\Vert u\Vert_{L^p(X)}+\mathcal{E}_p(u).
\]
Indeed, if $1\le p<\infty$, for each $u\in D^{1,p}(X)$ there is a unique (up to sets of $\mu$-measure zero) non-negative function $g_u$
that is the $L^p$-limit of a sequence of upper gradients of $u$ from $L^p(X)$ and so that for each upper gradient $g$ of $u$ we
have that $\Vert g_u\Vert_{L^p(X)}\le \Vert g\Vert_{L^p(X)}$. The functions $g_u$ belong to a larger class of ``gradients" of $u$,
called $p$-weak upper gradients, see for example~\cite{HajK, HKST, Sh} or~\cite{BBbook}.

For $1\le p<\infty$, the metric measure space $(X,d,\mu)$ is said to support a $p$-Poincar\'e inequality if there are constants
$C_P>0$ and $\lambda\ge 1$ such that for all $u\in D^{1,p}(X)$ and balls $B=B(x,r)$ in $X$, we have
\[
\jint_{B(x,r)}|u-u_B|\, d\mu\le C_P\, r\, \left(\jint_{B(x,\lambda r)}g_u^p\, d\mu\right)^{1/p}.
\]
It was shown in~\cite{HajK} that if $X$ is a length space, then we can take $\lambda=1$ at the expense of increasing the constant
$C_P$.

For the rest of this section, we consider $\Om$ to be a domain in a complete metric space $X$.
From~\cite[Section~7]{AS} we know that if $X$ is locally compact and equipped with a measure $\mu$ so that
$\mu$ is doubling and supports a $p$-Poincar\'e inequality for some $1\le p<\infty$, then the zero-extension of $\mu$ to
the completion $\overline{X}$ of $X$ also is doubling and supports a $p$-Poincar\'e inequality. Moreover, in this case,
$N^{1,p}(X)=N^{1,p}(\overline{X})$. We will exploit this property in this paper by using the identity
$N^{1,p}(\Omega)=N^{1,p}(\overline{\Omega})$ when the restriction of $\mu$ to $\Om$ is also doubling and supports
a $p$-Poincar\'e inequality.

\subsection{Besov spaces}\label{subsec2.4}

Consider a metric measure space $(Z,d,\nu)$.  For $0<\theta<1$ and $1<p<\infty$ we will consider the following
Besov energy:
\[
\Vert u\Vert_{\theta,p}^p:=
\int_{Z}\int_{Z}\frac{|u(y)-u(x)|^p}{d(x,y)^{\theta p}\, \nu(B(x,d(x,y)))}\, d\nu(y)\, d\nu(x),
\]
and set $B^\theta_{p,p}(Z)$ to be the space of all $L^p$--functions for which this energy is finite.
If $\nu$ is Ahlfors $Q$-regular, then
we can replace $\nu(B(x,d(x,y))$ with $d(x,y)^Q$ to obtain an equivalent energy. The following lemma
holds also for this modified norm.
The homogeneous Besov space $HB^\theta_{p,p}(\partial\Omega)$ is the collection of equivalence classes of functions from
$B^\theta_{p,p}(\partial\Omega)$, where two functions $u,v\in B^\theta_{p,p}(\partial\Omega)$ are said to be equivalent
if $\Vert u-v\Vert_{\theta,p}=0$, that is, $u-v$ is $\nu$-a.e.~constant on $\partial\Omega$. By a slight
abuse of notation, we conflate equivalence classes that are elements of $HB^\theta_{p,p}(\partial\Omega)$
with representative functions in those classes, but are careful to remember that then there is an ambiguity up to additive
constants here.

\begin{lem}\label{lem:Lp-BesovE}
Let $Z$ be a bounded metric space equipped with a measure $\nu$ with $\nu(Z)<\infty$.  If $f\in HB^\theta_{p,p}(Z)$, then $f\in L^p(Z)$.
Moreover, there exists $C>0$ such that for all $f\in HB^\theta_{p,p}(Z)$ we have that
\begin{equation}\label{eq:Lp-Bes}
\|f-f_Z\|_{L^p(Z)}\le C\|f\|_{\theta,p}.
\end{equation}
Here $f_Z=\jint_Z f\, d\nu$.
In particular, $HB^\theta_{p,p}(\partial\Omega)$ is a reflexive Banach space under the norm $\Vert \cdot\Vert_{\theta,p}$,
and is a Hilbert space when $p=2$.
\end{lem}

\begin{proof}
Since
\[
\|f\|^p_{\theta,p}=\int_Z\int_Z\frac{|f(y)-f(x)|^p}{\nu(B(y,d(x,y)))d(x,y)^{\theta p}}d\nu(y)d\nu(x)<\infty,
\]
there exists $x_0\in Z$ such that $|f(x_0)|<\infty$ and
\begin{align*}
C:=\int_Z\frac{|f(y)-f(x_0)|^p}{\nu(B(y,d(x_0,y)))d(x_0,y)^{\theta p}}d\nu(y)<\infty.
\end{align*}
Thus, we have that
\[\frac{1}{\nu(Z)\diam(Z)^{\theta p}}\int_Z|f(y)-f(x_0)|^pd\nu(y)\le C.
\]
Since $Z$ is bounded and $\nu(Z)<\infty$, it follows that
\begin{align*}
\int_Z|f(y)|^p&\le\int_Z\left(|f(y)-f(x_0)|+|f(x_0)|\right)^pd\nu(y)\\
	&\le 2^p\left(\int_Z|f(y)-f(x_0)|^pd\nu(y)+|f(x_0)|^p\nu(Z)\right)\\
	&\le 2^p\left(C\nu(Z)\diam(Z)^{\theta p}+|f(x_0)|^p\nu(Z)\right)<\infty. 
\end{align*}

To see that the homogeneous space is a Banach space, it suffices to show that it is complete under the given norm. Let
$\{u_k\}$ be a Cauchy sequence in $HB^\theta_{p,p}(Z)$. Replacing $u_k$ with $u_k-\jint_Zu_k\, d\nu$ if needed,
we may assume that $(u_k)_Z:=\int_Z u_k\, d\nu=0$ for each $k$. Then by~\eqref{eq:Lp-Bes} we know that
$\{u_k\}$ is a Cauchy sequence also in $L^p(Z)$, and hence converges to some function $u\in L^p(Z)$. By passing to a subsequence
if necessary, we may also assume that this convergence also occurs pointwise $\nu$-a.e.~in $Z$.
Note that $\Vert u_k\Vert_{\theta,p}=\Vert v_k\Vert_{L^p(Z\times Z,\nu_0)}$ where
$\nu_0$ is the weighted measure on $Z\times Z$ given by
\[
\nu_0(A)=\iint_A \ \frac{1}{d(x,y)^{\theta p+Q-1}}\, d\nu\times\nu(x,y),
\]
and
\[
v_k(x,y)=u_k(x)-u_k(y).
\]
It follows that $\{v_k\}$ is also a Cauchy sequence, and hence converges to some function $v\in L^p(Z\times Z,\nu_0)$.
By considering the corresponding subsequence and applying Fubini's theorem, we also know that
$v_k$ converges pointwise $\nu$-a.e.~in $Z$ to the function $(x,y)\mapsto u(x)-u(y)$. Hence we have the
desired conclusion that $\Vert u_k-u\Vert_{\theta,p}\to 0$ as $k\to \infty$ and $u\in HB^\theta_{p,p}(Z)$ as desired.
Since $1<p<\infty$, it is clear from the reflexivity of $L^p(Z\times Z,\nu_0)$ that $HB^\theta_{p,p}(Z)$ is also reflexive.

Since the norm on $HB^\theta_{2,2}(Z)$ is given via an inner product $\mathcal{E}_2$, and as the above argument tells
us that $HB^\theta_{2,2}(Z)$ is complete, we conclude that it is a Hilbert space.
\end{proof}

We will show below  that any doubling, compact metric measure space $(Z,d,\nu)$ arises as boundary
$Z=\partial\Om$, with $\Om$ a space satisfying the structure conditions (H0), (H1), (H2), and such that $\nu$
satisfies the comparison~\eqref{eq:Co-Dim-intro}.
Such co-dimension condition arises in the study of traces
of the $N^{1,p}$ and $D^{1,p}$--classes on $\Om$ to $\partial\Om$ (see for instance~\cite{BBS, M, MS}),
and is a natural condition that arises from considering
$\Om$ to be a uniformization of a hyperbolic filling of a doubling compact metric measure space as in~\cite{BBS}.
Besov spaces on $\partial\Om$ arise as trace classes of Newton-Sobolev spaces on $\Om$. The following result is
from~\cite{M}, but in the setting of $\Om$ such a trace result can also be found in~\cite{BBS}.

We equip the Besov space $B^\theta_{p,p}(\partial\Om)$ with the norm
\[
\Vert u\Vert_{B^\theta_{p,p}(\partial\Om)}:=\Vert u\Vert_{L^p(\partial\Om)}+\Vert u\Vert_{\theta,p}.
\]

\begin{thm}[{\cite[Theorem~1.1]{M}}] \label{thm:Maly-Trace}
Let $\Om$ be a John domain in $X$, with $\overline{\Om}$ compact such that the restriction of $\mu$ to $\overline{\Om}$
is doubling and supports a $p$-Poincar\'e inequality. Suppose in addition that $\partial\Om$ is equipped with a measure
$\nu$ that satisfies~\eqref{eq:Co-Dim-intro}. Then there is a bounded linear trace operator
$T:N^{1,p}(\Om)\to B^{1-\Theta/p}_{p,p}(\partial\Om)$
and a bounded linear extension operator $E:B^{1-\Theta/p}_{p,p}(\partial\Om)\to N^{1,p}(\Om)$ such that
\begin{enumerate}
\item[{\rm (i)}] $T\circ E u=u$ for $u\in B^{1-\Theta/p}_{p,p}(\partial\Om)$,
\item[{\rm (ii)}] for each $u\in N^{1,p}(\Om)$, for $\mathcal{H}$-a.e.~$x\in\partial\Om$ we have
\[
\lim_{r\to 0^+}\jint_{B(x,r)\cap\Om}|u-Tu(x)|\, d\mu=0.
\]
\end{enumerate}
\end{thm}

The domains $\Om$ considered in this paper satisfy the hypotheses given in Theorem~\ref{thm:Maly-Trace}.
Note that if $\Om$ is a uniform domain in $X$ and $(X,d,\mu)$ is doubling and supports a $p$-Poincar\'e inequality, then
the restrictions of $\mu$ to $\Om$ and to $\overline{\Om}$ are doubling and support a $p$-Poincar\'e inequality, see
for example~\cite[Theorem~4.4]{BjS}.

\subsection{Cheeger differential structures}

We now describe Cheeger (linear) differential structures in doubling metric measure spaces
$(X,d,\mu)$; these structures play a key role in the definition of Cheeger $p$-harmonic functions.

We say that a metric measure space $(X,d,\mu)$ supports a Cheeger differential structure if there is a positive integer
$N$ and a collection
$\{X_\alpha\}_{\alpha\in A}$, with each $X_\alpha$ a measurable subset of $X$, such that $\mu(X_\alpha)>0$ and
$\mu(X\setminus\bigcup_\alpha X_\alpha)=0$, an inner-product structure $\langle\cdot,\cdot\rangle_x$, $x\in X_\alpha$,
on $\R^n$ that is $\mu$-measurable,
and for each $\alpha\in A$ a Lipschitz map $\pip_\alpha:X_\alpha\to\R^N$ satisfying the
following condition for each Lipschitz function $f:X\to\R$: for $\mu$-a.e.~$x\in X_\alpha$ there is a vector
$\nabla f(x)\in \R^N$ such that
\[
\text{ess lim sup}_{X_\alpha\ni y\to x}\ \frac{|f(y)-f(x)-\langle \nabla f(x),\pip_\alpha(y)-\pip_\alpha(x)\rangle_x|}{d(y,x)}=0.
\]
For doubling spaces we can assume that $A$ is a countable set and that the collection $(X_\alpha)_{\alpha\in A}$ is
pairwise disjoint.
For doubling metric measure spaces $X$ supporting a $p$-Poincar\'e inequality,
such a differential structure was constructed by Cheeger in~\cite{Che}; the structure
constructed there satisfies the additional property that there is a constant $C\ge 1$ such that whenever $g_f$ is a minimal
$p$-weak upper gradient of a Lipschitz function $f$, then $C^{-1}g_u(x)\le \langle \nabla f(x),\nabla f(x)\rangle_x\le C\, g_u(x)$ for
$\mu$-a.e.~$x\in X$. In this paper we will consider such a differential structure. From the discussion in~\cite{Che} we know
that the notion of $\nabla f$ extends from the class of locally Lipschitz functions in $X$ to functions in
$D^{1,p}(X)$; see also the discussion in~\cite{FHK}. We set $|\nabla f(x)|=\langle \nabla f(x),\nabla f(x)\rangle_x^{1/2}$.

We note that there is more than one possible Cheeger differential structure on $X$, leading to us considering a wide range
of differential operators, one for each such structure. We say that a function $u\in D^{1,p}(\Om)$ is a \emph{Cheeger $p$-harmonic
function} in $\Om$ if, whenever $v\in D^{1,p}(\Om)$ has compact support in $\Om$, we have
\[
\int_{\text{supt}(v)}|\nabla u|^p\, d\mu\le \int_{\text{supt}(v)}|\nabla (u+v)|^p\, d\mu.
\]
Equivalently, we have the following corresponding Euler-Lagrange equation:
\[
\int_\Om|\nabla u(x)|^{p-2}\langle \nabla u(x),\nabla v(x)\rangle_x\, d\mu(x)=0.
\]
For brevity, in our exposition we will suppress the dependence of $x$ on the inner product structure, and denote
\[
\langle \nabla u(x),\nabla v(x)\rangle_x=\nabla u(x)\cdot\nabla v(x)
\]
when this will not lead to confusion.
Cheeger $p$-harmonic functions are quasiminimizers in the sense of Giaquinta, and hence
we can avail ourselves of the properties derived in~\cite{KS}.

\subsection{Neumann boundary value problem}

In~\cite{MS} a generalization of the Neumann boundary value problem
\begin{align*}
\Delta_p u=0 &\text{ on }\Om,\\
 |\nabla u|^{p-2}\partial_\eta u=f &\text{ on }\partial\Om
\end{align*}
was constructed and analyzed. There 
it was assumed that the boundary data $f$ is a bounded measurable
function on $\partial\Om$ such that $\int_{\partial\Om} f\, d\nu=0$.
Strictly speaking, the version considered in~\cite{MS} is the problem of minimzing the energy operator
\[
N^{1,p}(\Om)\ni v\mapsto I_f(v):=\int_{\overline{\Om}} g_v^p\, d\mu-p\int_{\partial{\Om}}vf\, d\nu,
\]
with $g_v$ the minimal $p$-weak upper gradient of $v$ and $f\in L^\infty(\overline{\Om}, \nu)$ in addition.
However, the proofs given in~\cite{MS}
are robust and apply also to the Cheeger differential formulation considered in the current paper, and we will use the results
from~\cite{MS} here, namely the boundedness property of solutions to the Neumann problem.
Existence of solutions was established in~\cite{MS},
and it was also shown that solutions are bounded.
In Appendix~\ref{Sec:Append}, we show that the boundedness
of the Neumann data $f$ is not needed in order to obtain existence and boundedness of solutions.
In fact, we have the following result as a consequence of the discussion in
Lemma~\ref{lem:bdd-MalySh} from Section~\ref{Sec:Append} together with the results from~\cite{MS} (for existence
of the minimizers).

\begin{thm}\label{th:Neumann-reg}
Let $\Om$ satisfy the structural assumptions~(H0), (H1), and~(H2). Then for each $f\in L^{p'}(\partial\Om)$
with $\int_{\partial\Om}f\, d\nu=0$ there is a bounded function $u\in N^{1,p}(\Om)$ such that
whenever $v\in N^{1,p}(\Om)$, we have
\[
\int_\Om|\nabla u|^p\, d\mu-p\int_{\partial\Om}u\, f\, d\nu
   \le \int_\Om|\nabla v|^p\, d\mu-p\int_{\partial\Om}v\, f\, d\nu.
\]
Moreover, 
we can choose $u$ so that $\int_\Om u\, d\mu=0$.
\end{thm}

 \subsection{Hyperbolic fillings}\label{sub:hyp-fill}

 In this subsection we give a brief description of the hyperbolic filling of a compact doubling metric measure space
 as given in~\cite{BBS}. While this construction differs somewhat from the constructions given in earlier literature,
 in essence the metric portion of the construction is similar to that of~\cite{BSak, BSakSou, BSch, BourPaj}.

 With $(Z,d)$ a compact metric space equipped with a doubling measure $\nu$, we fix $\alpha>1$, $\tau>1$, and, for
 each non-negative integer $n$ we choose a maximal $\alpha^{-n}$-separated set $S_n\subset Z$. By scaling the metric
 if need be, we can always assume that the diameter of $Z$ is smaller than $1$; hence $S_0$ contains only one point $x_0$.
 We can also ensure that $S_n\subset S_{n+1}$, and consider the vertex set $V=\bigcup_{n=0}^\infty S_n\times\{n\}$.
 Two points $(x,n)$ and $(y,m)$ are declared to be neighbors if either $B_Z(x,\alpha^{-n})$ intersects $B_Z(y,\alpha^{-m})$
 with $|n-m|=1$, or if $B_Z(x,\tau\alpha^{-n})\cap B_Z(y,\tau\alpha^{-m})$ is non-empty with $n=m$; here $B_Z$ denotes
a ball in the metric space $Z$. This converts $V$ into
 a graph $X$, with each edge assigned a unit length interval. As shown in~\cite{BBS}, the graph $X$, equipped with the path
 metric, is a roughly star-like Gromov-hyperbolic space, which, when uniformized via the metric
 \[
 d_\eps(v,w)=\inf_\gamma\int_\gamma e^{-\eps d(\gamma(t), v_0)}\, ds(t), \ v,w\in X,
 \]
where $v_0=(x_0,0)$ and the infimum is over all curves in $X$ with end points $v,w$, turns the metric
graph $(X,d)$ into a uniform domain $X_\eps:=(X,d_\eps)$ when $\eps=\log(\alpha)$.
The metric graph has a natural measure on it, given by considering the one-dimensional Hausdorff measure $\mathcal{H}^1$ on the
edges of the graph. Next, following~\cite{BBS}, for
each $\beta>0$, one can lift the measure
$\nu$ on $Z$  up to a measure $\mu_\beta$ in $X$ by setting for each Borel set $A\subset X$,
\[
\mu_\beta(A):=\int_A e^{-\beta d(x,v_0)} \nu(B_Z(\zeta_x, \alpha^{-n_x}))\, d\mathcal{H}^1(x),
\]
where $(\zeta_x, n_x)\in V$ is a nearest vertex to the point $x\in X$. It is shown that with $\eps=\log(\alpha)$ and
$\beta>0$, the metric measure space $(X,d_\eps,\mu_\beta)$ is doubling and supports a $1$-Poincar\'e inequality; moreover,
$Z$ is biLipschitz to the boundary $\partial_\eps X$ of the uniform domain $(X,d_\eps)$, and in addition,
$\Om:=X$ equipped with the metric $d_\eps$ and measure $\mu_\beta$ satisfies our structural
conditions~(H0), (H1), and~(H2) with $\Theta=\beta/\eps$. As shown in~\cite{BBS}  the trace class of the Sobolev
space $N^{1,p}(X,d_\eps,\mu_\beta)$ is the Besov class $B^\theta_{p,p}(Z)$, where $\theta=1-\beta/(\eps p)=1-\Theta/p$.

Thus, with our primary object $(Z,d,\nu)$, for each $1<p<\infty$ and $0<\theta<1$, we can choose $\beta=p\eps(1-\theta)$
and $\eps=\log(\alpha)$
in the above construction to obtain a uniform domain $\Om=(X,d_\eps)$ equipped with the measure $\mu_\beta$
that satisfies our structural assumptions and yields the fractional Laplacian $(-\Delta_p)^\theta$.

\section{Equivalent Formulations of the Neumann problem for the $p$-Laplacian}\label{Sect:equiv-form}

In this section we prove Theorem~\ref{thm:equiv-intro}.
Formulation~(a), when seen as a problem on the
independent metric measure space $\overline{\Omega}$
corresponds to the inhomogeneous problem $-\Delta_p u=\nu$ with $\nu$ a (signed) Radon measure on the space of
interest such that $\nu$ is in the dual of the Sobolev space $W^{1,p}$; see for example~\cite{KM,Mi, RZ, Z}.
The formulation~(b) (with $|\nabla v|$ and with $|\nabla v|$ replaced by the minimal $p$-weak upper gradient $g_v$)
was considered in~\cite{MS}. Formulation~(c) was motivated by the study of nonlocal minimization problems considered
in~\cite{CS} (Euclidean setting), \cite{FF} (Carnot group setting), \cite{BGS, CG} (manifold setting with $\gamma=1/2$)
and~\cite{EbGKSS} (metric spaces of controlled geometry).

\begin{proof}[Proof of Theorem~\ref{thm:equiv-intro}]
We first show that~(a) is equivalent to~(b). To do so, first suppose that~(b) holds for $u$, and let $\phi\in N^{1,p}(\overline{\Omega})$.
Then setting for each $\eps\in\R$,
\[
I(\eps):=\int_\Omega|\nabla (u+\eps \phi)|^p\, d\mu-p\int_{\partial{\Omega}}(u+\eps \phi)\, f\, d\nu,
\]
we know that $I(\eps)$ has a minimum at $\eps=0$. Therefore $\tfrac{d}{d\eps}I\vert_{\eps=0}=0$.
Note that
\[
\frac{d}{d\eps}I(\eps)=\frac{p}{2}\int_\Omega |\nabla (u+\eps \phi)|^{p-2}\, 2\, \nabla(u+\eps\phi)\cdot\nabla \phi\, d\mu
  -p\int_{\partial{\Omega}}\phi\, f\, d\nu.
\]
Letting $\eps=0$ gives
\[
\frac{d}{d\eps}I(\eps)\vert_{\eps=0}\,
=p\, \int_\Omega|\nabla u|^{p-2}\nabla u\cdot\nabla \phi\, d\mu-p\int_{\partial{\Omega}}\, \phi\, f\, d\nu.
\]
Setting this equal to zero yields~(a).

Next, suppose that~(a) holds for $u$. We will use the following convexity of $|\cdot|^p$ for vectors when $p>1$: when $\xi,\eta\in\R^N$,
and $\langle\cdot,\cdot\rangle$ is an inner product on $\R^N$, we have
\[
 |\xi|^p-|\eta|^p\ge p|\eta|^{p-2}\langle\eta,\xi-\eta\rangle
\]
where $|\cdot|$ is the norm corresponding to this inner product. Therefore when $v\in N^{1,p}(\overline{\Omega})$,
\[
\int_\Omega\left(|\nabla v|^p-|\nabla u|^p\right)\, d\mu\ge p\int_\Omega|\nabla u|^{p-2}\nabla u\cdot\nabla (v-u)\, d\mu.
\]
Applying~(a) to the function $\phi=v-u$ gives the desired inequality $I(u)\le I(v)$.

It now only remains to show that~(a)+(b) is equivalent to~(c). This is a Morrey-type argument.
For each $\eps>0$ let $\eta_\eps$ be as in the statement of~(c).
Suppose that $u$ satisfies~(a). Let $\phi$ be a Lipschitz function on $\overline{\Omega}$. Then by~(a) we have that
\[
\int_{\partial\Omega}\phi\, f\, d\nu=\int_\Omega|\nabla u|^{p-2}\nabla u\cdot\nabla \phi\, d\mu
  =\int_\Omega|\nabla u|^{p-2}\nabla u\cdot\nabla((1-\eta_\eps)\phi)\, d\mu,
\]
where the latter equality follows from the fact that $\eta_\eps\phi=0$ on $\partial\Omega$ and by~(a) applied to $\eta_\eps\phi$.
Note that the Cheeger differential structure also follows the Leibniz rule
(see for example~\cite[(4.43)]{Che} or~\cite{G}), and so
$\nabla(1-\eta_\eps)\phi=(1-\eta_\eps)\nabla\phi-\phi\,\nabla\eta_\eps$. Since $|\nabla u|\in L^p(\Omega)$, we have that
\[
\lim_{\eps\to 0^+}\int_\Omega (1-\eta_\eps)\, |\nabla u|^{p-2}\nabla u\cdot\nabla\phi\, d\mu=0.
\]
It follows that
\[
\int_{\partial\Omega}\phi\, f\, d\nu=-\lim_{\eps\to 0^+}\int_\Omega\phi\, |\nabla u|^{p-2}\nabla u\cdot\nabla\eta_\eps\, d\mu.
\]
As the above holds for all Lipschitz $\phi$ on $\overline{\Omega}$, the claim~(c) follows.

Finally, suppose that $u$ satisfies~(c). The above argument with $\phi=\eta_\eps\phi+(1-\eta_\eps)\phi$ gives that
\[
\int_\Omega|\nabla u|^{p-2}\nabla u\cdot\nabla(\eta_\eps \phi)\, d\mu=0
\]
because of the $p$-harmonicity of $u$ in $\Omega$, and so for Lipschitz $\phi$ on $\overline{\Omega}$,
\[
\int_\Omega|\nabla u|^{p-2}\nabla u\cdot\nabla \phi\, d\mu=-\lim_{\eps\to 0^+}\int_\Omega\phi\, |\nabla u|^{p-2}\nabla u\cdot\nabla\eta_\eps\, d\mu
   =\int_{\partial\Omega}\phi\, f\, d\nu,
\]
that is, (a) holds true for Lipschitz $\phi$. Now, if $v\in N^{1,p}(\overline{\Omega})$, then by the fact that
$\overline{\Omega}$ supports a $p$-Poincar\'e inequality, there is a sequence of Lipschitz functions $\phi_k$ on
$\overline{\Omega}$ such that $\phi_k\to v$ in $N^{1,p}(\overline{\Omega})$, see~\cite[Theorem~8.2.1]{HKST}.
Then as $\nabla \phi_k\to \nabla v$ in $L^p(\Omega;\R^N)$, we have
\[
\lim_{k\to\infty}\int_\Omega|\nabla u|^{p-2}\nabla u\cdot\nabla \phi_k\, d\mu
=\int_\Omega|\nabla u|^{p-2}\nabla u\cdot\nabla v\, d\mu.
\]
Moreover, as $\phi_k\to v$ in $N^{1,p}(\overline{\Omega})$ and trace
$T:N^{1.p}(\overline{\Omega})\to B^\theta_{p,p}(\partial\Omega)\subset L^p(\partial\Omega,\, \nu)$ is a bounded
linear map with $f$ a bounded Borel function on $\partial \Omega$, we obtain also that
\[
\lim_{k\to\infty}\int_{\partial\Omega}\phi_k\, f\, d\nu=\int_{\partial\Omega}v\, f\, d\nu.
\]
Thus~(a) follows for \emph{all} $v\in N^{1,p}(\overline{\Omega})$.
\end{proof}

\section{Boundary Regularity for $1<p\le Q$}\label{Sect:boundary-reg}

The goal of this section is to prove Theorem~\ref{thm:holder-intro}.
We
will only consider the range $1<p\le Q$, since in the range $p>Q$  we already know that $u$ is $1-Q/p$-H\"older continuous on
$\overline{\Omega}$ in view of  the
Morrey embedding theorem, see~\cite[Lemma~9.2.12]{HKST} for instance.

Let $(\overline{\Omega}, d, \mu)$
be a metric measure space satisfying  the structural assumptions (H1), (H2) in the introduction.
In the following, for $x\in \overline\Omega$ and $r>0$ the balls in the induced metric are $B(x,r)=\{y\in \overline \Omega\ |\  d(y,x)<r\}$.
We recall from Subsection~\ref{subsec2.4} that in view of \cite[Theorem~1.3]{M}, there exists a bounded linear trace operator
$T:N^{1,p}(\Omega)\to B^{1-1/p}_{p,p}(\partial\Omega)\subset L^p(\partial\Omega).$
For ease of notation, we will denote $Tu$, which is a function on $\partial\Omega$, also by $u$.
We recall from the discussion at the end of Subsection~\ref{subsect:2.2} that
because $(\overline{\Om},d,\mu\vert_\Om)$ supports a $p$-Poincar\'e inequality and $\mu\vert_\Om$ is doubling, it follows that
the extension of $u$ by $Tu$ to $\partial\Om$ is in $N^{1,p}(\overline{\Om})$,
see for example~\cite{KL} or~\cite[Theorem~9.2.8]{HKST}.

Let $f:\partial\Omega\to\R$ be bounded and Borel measurable, with $\int_{\partial\Omega}f\,d\nu =0$
and $f\in L^{p'}(\partial\Om)$.
Suppose that $u\in N^{1,p}(\overline{\Omega})$ is  $p$-harmonic
with Neumann boundary conditions; that is, for all $\phi\in N^{1,p}(\overline{\Omega})$,
\[
\int_{\Omega}|\nabla u|^{p-2}\nabla u\cdot\nabla\phi\,d\mu=\int_{\partial\Omega}\phi f\,d\nu.
\]
Then by~\cite{MS} or by Theorem~\ref{thm:equiv-intro}, $u$ is a minimizer of
$\int_{\Omega}|\nabla u|^{p}\,d\mu-p\int_{\partial\Omega}u\,  f\,d\nu$.
We also require that $\int_\Om u\, d\mu=0$. This is always possible by subtracting a constant from any given
solution. Note that as we deal with a choice of Cheeger differential structure, such a solution is unique,
see~\cite[Lemma~4.5 and the subsequent comment]{MS}.

For $x\in \overline \Om$ and $r>0$ the   boundary of $B(x,r)$ relative to the topology induced by the metric
on $\overline \Om$, is $\partial B(x,r)\subset \{y\in \Omega \ | d(x,y)=r\}$.

We begin by recalling the notion of $p$-harmonic extension and an immediate application
of some interior H\"older estimates established in \cite[Theorem~5.2]{KS}.

\begin{prop}\label{prop:KS}
Let $u\in N^{1,p}(\overline \Om)$, and $x,r$ as above. There exists a unique function $v\in N^{1,p}(\overline{\Om})$ such that
$v$ is $p$-harmonic in $B(x,r)$ and $v=u$ on $\overline{\Om}\setminus B(x,r)$.
Moreover such $v$ is H\"older continuous in $B(x,r/2)$,
and it satisfies the estimate
\[
\sup_{B(x,r/2)} |v-v(x)|\le C\|u\|_{L^\infty}\, r^{\alpha}
\]
for some constants $C,\alpha$ depending only on the structure conditions of $\Om$.
\end{prop}

\begin{remark}
We want to emphasize that the function $v$ is the solution to the Dirichlet problem on the domain $\overline{\Om}\cap B(x,r)$,
whose boundary  is $\overline{\Om}\cap \partial B(x,r)$. In other words, for all $\phi \in N^{1,p}_0(\overline{\Om}\cap B(x,r))$ one has
$$
\int_{\overline{\Om}\cap B(x,r)} |\nabla v|^{p-2} \nabla v \cdot\nabla \phi \,d\mu = 0,
$$
whereas for the solution of the Neumann problem $u$ we have instead
$$
\int_{\overline{\Om}\cap B(x,r)} |\nabla u|^{p-2} \nabla u \cdot\nabla \phi \,d\mu = \int_{\partial \Omega \cap B(x,r)} \phi f d\nu.
$$
Because of this consideration, the  interior H\"older estimates in~\cite{KS}
can be applied to $v$  in the metric space $X=\overline{\Om}$ and the domain $B(x,r)\cap\overline{\Om}$, thus
establishing H\"older continuity of $v$ in $B(x,r)\cap\overline{\Om}$.

Invoking interior estimates to prove regularity up to the boundary may give pause to  readers familiar with the
smooth setting. The issue here is that the Dirichlet problem for $v$ in $B(x,r)\cap\overline{\Om}$  is not the same
as a Dirichlet problem in $B(x,r)\cap{\Om}$. The  test functions used in the former are in
$N^{1,p}_0(B(x,r)\cap\overline{\Om})$ and thus do not need to have zero trace on $\partial \Omega \cap B(x,r)$.
In fact  the function $v$ does not satisfy the Dirichlet problem (with boundary data $u$)
in  $\Om\cap B(x,r)$, since there is an
extra part of the boundary, namely, $B(x,r)\cap\partial\Om$. On this extra part of the
boundary, the interested reader should note that, in the smooth Riemannian setting, $v$ satisfies  zero Neumann boundary condition.
A similar perspective can be found in the discussion on orbifolds~\cite{BW}.
\end{remark}

The following proposition gives Morrey type bounds on the growth of the averages of  $|\nabla u|^p$ on balls near points
on the boundary $\partial\Omega$.
This bound will be used in Theorem~\ref{thm:holder-intro} to show H\"{o}lder continuity
of $u$ near these points. Recall that the measure $\nu$ is $\Theta$-codimensional with respect to $\mu$ for some
$0<\Theta\le 1$, see~\eqref{eq:Co-Dim-intro} above.

\begin{prop}\label{thm:bounds}
 Suppose that $f\in L^q(\partial\Om)\cap L^{p'}(\partial\Om)$ for some $q>1$.
There exists a constant $\mathbb M$,
depending only on $q$ and
the structure conditions (H1), (H2), such that
when $u$ is a solution to the Neumann boundary value problem with boundary data $f$,  then
for every $x_0\in\partial\Omega$, 
and every $x\in\overline{\Omega}$ and $r>0$ such that $B(x,2r)\subset B(x_0,R)$ with $R<\text{diam}(\Omega)/2$, we have
\[
\int_{B(x,r/4)\cap\Omega}|\nabla u|^{p}\,d\mu
\le \mathbb{M}\left[M^p\,\frac{\mu(B(x,r))}{r^{(1-\alpha)p}}
+M\left(\int_{B(x_0,R)\cap \partial \Om}|f|^q\, d\nu\right)^{1/q}\frac{\mu(B(x,r))^{1/q'}}{r^{\Theta/q'}}\right].
\]
Here, $q'=q/(q-1)$,  $\alpha$ is as in Proposition \ref{prop:KS},
 and $M=\sup_\Om |u|$.
\end{prop}

For ease, we adopt the following convention in the proof: when notating the integral over a ball $B$ with respect to $\mu$ and $\nu$, we mean integration over $B\cap\Omega$ and $B\cap\partial\Omega$ respectively.

\begin{proof}
Consider a $p$-harmonic extension $v$ of
$u$ from $\overline{\Omega}\setminus B(x,r)$ to $B(x,r)\cap \Om$, as in Proposition \ref{prop:KS}.
Then $v\in N^{1,p}(\overline{\Omega})$ with $u-v\in N^{1,p}_0(B(x,r))$.

First, from the convexity of the function $t\mapsto t^p$ for $p>1$,
\begin{equation}
  \int_{B(x,r/4)}|\nabla u|^p\,d\mu
  \le 2^{p-1} \left(\int_{B(x,r/4)}|\nabla (u-v)|^p\,d\mu+\int_{B(x,r/4)}|\nabla v|^p\,d\mu \right). \label{eq:boundDup}
\end{equation}

Let us first estimate the first integral on the right hand side. The $p$-Laplace operator satisfies the following structural conditions:
\begin{equation}\label{eqn:monotonicity}
(|z|^{p-2}z-|w|^{p-2}w)\cdot(z-w)\ge
\begin{cases}
 C'|z-w|^p, & p\ge 2\\
 C'' (|z|+|w|)^{p-2}|z-w|^2, & p\le 2.
\end{cases}
\end{equation}
Here the constants $C'$ and $C''$ only depend on $p$.

When $p\ge 2$, the use of \eqref{eqn:monotonicity}  gives us
\[
 \int_{B(x,r/4)}|\nabla (u-v)|^{p}\,d\mu \le C\int_{B(x,r/4)}(|\nabla u|^{p-2}\nabla u-|\nabla v|^{p-2}\nabla (v))\cdot(\nabla u-\nabla v)\,d\mu.
\]
Combining this with~\eqref{eq:boundDup} gives
\begin{align}\label{eqn:bigP}
\int_{B(x,r/4)}|\nabla u|^p\, d\mu &\le C\int_{B(x,r/4)}(|\nabla u|^{p-2}\nabla u-|\nabla v|^{p-2}\nabla (v))\cdot(\nabla u-\nabla v)\,d\mu
\notag\\
  &\hskip 7cm +C\int_{B(x,r/4)}|\nabla v|^p\, d\mu.
\end{align}

For $p<2$, we first use Young's inequality with exponents $2/p$ and $2/(2-p)$ and then \eqref{eqn:monotonicity} to obtain
\begin{align}\label{eqn:plessthan2}
\int_{B(x,r/4)} &|\nabla (u-v)|^p \,d\mu\notag\\
= & \int_{B(x,r/4)}\tau^{(p-2)/2}|\nabla (u-v)|^p(|\nabla u|+|\nabla v|)^{p(p-2)/2}\cdot \tau^{(2-p)/2}(|\nabla u|
    +|\nabla v|)^{p(2-p)/2}\,d\mu\notag\\
\le & \int_{B(x,r/4)} \frac{p\, \tau^{(p-2)/p}}{2} |\nabla (u-v)|^2(|\nabla u|+|\nabla v|)^{p-2}
        + \frac{\tau (2-p)}{p}(|\nabla u|+|\nabla v|)^p \,d\mu\notag\\
\le & C \tau^{(p-2)/p}\int_{B(x,r/4)}\left(|\nabla u|^{p-2}\nabla u-|\nabla v|^{p-2}\nabla v\right)\cdot(\nabla u-\nabla v)\, d\mu\notag\\
     &\hskip 8cm+C\tau\int_{B(x,r/4)}\left(|\nabla u|^p+|\nabla v|^p\right)\, d\mu.
\end{align}

By choosing $\tau$ small enough so that $C\tau\leq 1/2$ we can absorb the integral of $|\nabla u|^p$ to the
left hand side of ~\eqref{eq:boundDup} to obtain~\eqref{eqn:bigP} even for the case $1<p<2$, that is,
\begin{align}\label{eqn:boundDup2}
\nonumber \int_{B(x,r/4)}|\nabla u|^p\,d\mu
 &\leq C \int_{B(x,r)}(|\nabla u|^{p-2}\nabla u-|\nabla v|^{p-2}\nabla v)\cdot(\nabla u-\nabla v)\,d\mu\\
 &\hskip 7cm+C\int_{B(x,r/4)}|\nabla v|^p \,d\mu.
\end{align}
Notice that all the integrands are nonnegative and thus integrating over  $B(x,r)$ instead of $B(x,r/4)$
can only increase the integrals.

Let us first consider the first term on the right hand side. Using $u-v$ as a test function,
we have from the weak formulation of the $p$-Laplacian equation~\ref{thm:equiv-intro}(a)
(with $\Omega$ replaced by $B(x,r)$ and $f$ replaced by the constant function $0$ on $B(x,r)$; recall that
$v$ is a solution to the Dirichlet problem on $B(x,r)\cap\overline{\Om}$ with boundary data $u$ on
$\Om\cap\partial B(x,r)$) that
\[
\int_{B(x,r)}|\nabla v|^{p-2}\nabla v\cdot \nabla (u-v)\,d\mu=0.
\]
Hence, integrating over the ball $B(x,r)$ we get
\begin{align*}
 \nonumber
 &\ \int_{B(x,r)}(|\nabla u|^{p-2}\nabla u-|\nabla v|^{p-2}\nabla v)\cdot(\nabla u-\nabla v)\,d\mu \\
  &= \int_{B(x,r)}|\nabla u|^{p-2}\nabla u\cdot (\nabla u-\nabla v)-|\nabla v|^{p-2}\nabla v\cdot (\nabla u-\nabla v)\,d\mu \\
   &= \int_{B(x,r)}|\nabla u|^{p-2}\nabla u\cdot (\nabla u-\nabla v)\,d\mu.
\end{align*}
Now, since $u$ is a solution to the Neumann problem with boundary data $f$, we obtain
\[
\int_{B(x,r)}|\nabla u|^{p-2}\nabla u\cdot \nabla (u-v)\,d\mu=\int_{B(x,r)}(u-v)f\,d\nu.
\]
Hence, invoking the $L^\infty$ bounds on solutions, and the regularity assumption  (H2) of the measure $\nu$, one has

\begin{align}
\int_{B(x,r)}(u-v)f\,d\nu
  &\le \int_{B(x,r)}(|u|+|v|)|f|\,d\nu \notag \\
  &\le 2CM\int_{B(x,r)}|f|\,d\nu \notag \\
 &\le 2CM \nu(B(x,r))^{1/q'}\left(\int_{B(x,r)\cap \partial \Om}|f|^q\, d\nu\right)^{1/q} \notag \\
 &\le 2CM\left(\int_{B(x_0,R)} |f|^q\, d\nu\right)^{1/q}\, \frac{\mu(B(x,r))^{1/q'}}{r^{\Theta/q'}}.
 \label{eq:boundDu_vp}
\end{align}
Note that the constant C above changes from line to line, but only depends on
$p$ and the regularity constant of the measure $\nu$. Here we have used the fact that $u$, and hence by
maximum principle, $v$ are bounded in $B(x_0,R)$, see~\cite{MS}. An upper bound for $|u|+|v|$ in $B(x_0,R)$ is
denoted by $M$.

We now consider the second term of the sum on the right hand side of equation~\eqref{eqn:boundDup2}.

Since $v$ is $p$-harmonic in $B(x,r)$, then it is in the De Giorgi class
$DG_p(B(x,r))$, see~\cite[Proposition~3.3]{KS}.  Hence for all $k\in \R$,
\[
\int_{B(x,r/4)}|\nabla (v-k)_\pm|^{p}\,d\mu\le\frac{C}{(r/2-r/4)^p}\int_{B(x,r/2)}(v-k)_\pm^{p}\,d\mu.
\]
Here, for a function $h$, the function $h_+=\max\{h,0\}$ is the positive part of $h$ and $h_-=\max\{-h,0\}$ is
the negative part of $h$.
Hence, choosing $k=v(x)$, we have
\[
\int_{B(x,r/4)}|\nabla (v-v(x))_+|^{p}\,d\mu\le\frac{C}{r^p}\int_{B(x,r/2)}(v-v(x))_+^{p}\,d\mu
\]
and
\[
\int_{B(x,r/4)}|\nabla (v-v(x))_-|^{p}\,d\mu\le\frac{C}{r^p}\int_{B(x,r/2)}(v-v(x))_-^{p}\,d\mu.
\]
Summing these two inequalities, and recalling that $B(x,r/2)\subset B(x,\frac{3}{2}r)\subset \Om$, we can then invoke the scale invariant {\it local} $\alpha$-H\"older continuity estimates for $p$-harmonic functions from Proposition \ref{prop:KS}
(originally in \cite[Theorem~5.2]{KS}), obtaining
\begin{align}
  \int_{B(x,r/4)}|\nabla v|^{p}\,d\mu \le \frac{C}{r^p}\int_{B(x,r/2)}|v-v(x)|^{p}\,d\mu
   &\le \frac{C\,M^p}{r^p}\int_{B(x,r/2)}r^{\alpha p}\,d\mu  \notag \\
   &\le
   C M^p\,\frac{\mu(B(x,r))}{r^{(1-\alpha)p}}.
   \label{eq:boundDvp}
\end{align}
Combining inequalities~\eqref{eq:boundDup},~\eqref{eq:boundDu_vp}, and~\eqref{eq:boundDvp}, we finally conclude
\[
\int_{B(x,r/4)}|\nabla u|^{p}\,d\mu
\le C\, M^p \frac{\mu(B(x,r))}{r^{(1-\alpha)p}}+2CM\left(\int_{B(x_0,R)\cap \partial\Omega}|f|^q\, d\nu\right)^{1/q}\frac{\mu(B(x,r))^{1/q'}}{r^{\Theta/q'}}.\qedhere
\]
\end{proof}

Next, we establish the global  H\"{o}lder continuity of $u$ in $\overline\Omega$. The argument is
along the lines of the standard proof of Morrey embedding theorem found in~\cite{HajK}, and is a streamlined version
of the classical regularity proof found in PDE texts such as~\cite{Gia1, Gia2}. This classical proof is in
two parts: the first part is to show that functions whose gradient exhibit a decay property in the spirit of
Proposition~\ref{thm:bounds} belong to a Campanato space (see~\cite[page~43]{Gia2}) by using Poincar\'e inequalities,
and the second part is to show that functions in the Campanato space are locally H\"older continuous
(see~\cite[page~41]{Gia2}) by using a telescoping sequence of balls and a Lebesgue point argument. Both
parts are combined into one seamless argument in the proof given below for the convenience of the reader.

Now we are ready to prove Theorem~\ref{thm:holder-intro}.

\begin{proof}[Proof of Theorem~\ref{thm:holder-intro}]
Choose $R_0$ with $0<R_0<\text{diam}(\Omega)/2$.
For $x_0\in\partial \Omega$, let $x, y\in B(x_0,R_0/64)\cap \Om$ be $\mu$-Lebesgue points of $u$.
Then $d(x,y)<R_0/32$, $B(x, R_0/16)\subset B(x_0,R_0)$,
and $B(y,R_0/16)\subset B(x_0,R_0)$.

Define a sequence of balls indexed by $\Z$ centered at $x$ and $y$ as follows:
For $i\ge 0$, set $B_i=B(x,2^{1-i}d(x,y))$, and for $i<0$, set $B_i=B(y,2^{1+i}d(x,y))$.
We denote the radius of the ball $B_i$ by $\rho_i$; note that $\rho_i=2^{1-|i|}d(x,y)$.  For each ball
in the sequence, we can apply Proposition~\ref{thm:bounds} with $r=4\rho_i$.

By an application of the Poincar\'e inequality, the doubling condition and the Lebesgue
point property of $x,y$ for $u$ with respect to  the measure $\mu$,  we then estimate
\begin{align*}
  |u(x)-u(y)| \le \sum_{i\in\Z}|u_{B_i}-u_{B_{i+1}}|
   &\le C\sum_{i\in\Z}\jint_{B_i} |u-u_{B_i}|\,d\mu  \\
   &\le C\sum_{i\in\Z}\rho_i \left(\mu(B_i)^{-1}\int_{B_i} |\nabla u|^p\,d\mu\right)^{1/p}.
\end{align*}
Therefore, invoking Proposition~\ref{thm:bounds}, applied to $4B_i$ for each $i$, one obtains
\begin{align}
   |u(x)-u(y)|
   &\le C\sum_{i\in\Z}\frac{\rho_i}{\mu(B_i)^{1/p}}
   \left[\frac{\mu(B_i)^{1/p}}{\rho_i^{1-\alpha}}+\frac{\mu(B_i)^{1/(q'p)}}{\rho_i^{\Theta/(q'p)}}\right]\notag\\
   &= C\sum_{i\in\Z} \left[\rho_i^\alpha+\frac{\rho_i^{1-\Theta/(q'p)}}{\mu(B_i)^{(1-1/q')/p}}\right] \label{eq:applyboundsthm}\\
    \label{eq:applyboundsthm2} &\le C\sum_{i\in\Z} \left[\rho_i^\alpha+
     \left(\frac{R_0}{\mu(B(x_0,R_0))}\right)^{1/(pq)}\rho_i^{1-\Theta/(q'p)-Q/(qp)}\right]\\
   &\le C\sum_{i\in\Z}\rho_i^{1-\varepsilon}
   = Cd(x,y)^{1-\varepsilon}\sum_{i\in\Z}2^{-|i|(1-\varepsilon)}
   \ \le \  Cd(x,y)^{1-\varepsilon}. \notag
\end{align}
Here inequality~\eqref{eq:applyboundsthm2} follows from the  lower mass bound property of $\mu$,
see~\eqref{eq:lower-mass-exp}. The constant $C$ depends only on the structural constants as
well as on $\mu(B(x_0, R_0))$, $R_0$ and $\int_{B(x_0,R_0)}|f|^q\, d\nu$. This completes the proof of
H\"older's inequality.

Finally, to prove the Harnack inequality, we assume that $u\ge 0$ on $\Om$ and that $f=0$ on the
relatively open set $W\subset\partial\Om$.
By Theorem~\ref{thm:equiv-intro}(a) we have that whenever
$\pip\in N^{1,p}(\Om)=N^{1,p}(\overline{\Om})$ with support contained in $\Om\cup W$, we have
$\int_{\Om\cup W}|\nabla u|^{p-2}\nabla u\cdot\nabla\pip\, d\mu=0$, and hence it follows that
$u$ is Cheeger $p$-harmonic in the domain $\Om\cup W$, seen as a domain in the metric space $\overline{\Om}$.
By the strong maximum principle as in~\cite[Corollary~6.4]{KS},
we then have that either $u$ is identically zero in $\overline{\Om}$ (and hence satisfies the Harnack inequality
trivially), or else, $u>0$ in $\Om$. In this latter case,
invoking ~\cite[Corollary~7.3]{KS}, the desired Harnack inequality follows.
\end{proof}

\begin{remark} Although we were able to prove $L^\infty$ bounds for $u$ from the hypothesis $f\in L^{p'}(\partial \Om, \nu)$ (see the appendix), in the argument above we need the stronger integrability condition $f\in L^q(\partial \Om, \nu)$ with $Q-\Theta<(p-\Theta)q$.
If we only have that $f\in L^{p'}(\partial\Om)$ then the proof can be modified, so that one obtains a weaker regularity condition. Indeed, in this case, if $x,y\in\partial\Om$
such that the maximal function $M(f^{p'})$ is finite at those points, then
\[
|u(x)-u(y)|\le C d(x,y)^{1-\beta}\, \left[1+ M(f^{p'})(x)^{1/p'}+M(f^{p'})(y)^{1/p'}\right],
\]
where $\beta=\max\{1-\alpha, \Theta/p\}$, which is automatically smaller than $1$ as $p>\Theta$.
\end{remark}

\section{Stability of solutions under perturbation of the Neumann data}\label{Sect-stability}

In this section we prove two different stability results for  $p-$Laplacian Neumann problems in the metric measure
space setting, under $L^{p'}$ perturbations of the boundary data, with $p>1$ and $p'$ the H\"older conjugate of $p$.
The first is Theorem~\ref{thm:stability-cheeger-intro}. The second stability result corresponds to the
formulation of the Neumann problem using only minimal upper gradients, as  in ~\cite{MS} (i.e., the variational
formulation in Theorem~\ref{thm:equiv-intro} part (b), with $|\nabla u|$ substituted by the minimal upper gradient,) and it
is stated in  Theorem~\ref{thm:stability-uppergradient}.

We first begin by proving Theorem~\ref{thm:stability-cheeger-intro}.

\begin{proof}[Proof of Theorem~\ref{thm:stability-cheeger-intro}]
Let $f,g, u, v$ be as in the statement of the theorem.
Since $u$ and $v$ are solutions of the Neumann problems with boundary data $f$ and $g$, we
have that for all $\phi \in N^{1,p}(\bar \Omega)$,
\[
\int_\Omega |\nabla u|^{p-2} \nabla u\cdot \nabla \phi\, d\mu = \int_{\partial \Omega} \phi fd\nu
\text{ and }\int_\Omega |\nabla v|^{p-2} \nabla v\cdot \nabla \phi\, d\mu = \int_{\partial \Omega} \phi  g d\nu.
\]
Subtracting the second  identity from the first and substituting $\phi=u-v$, yields
\[
\int_\Omega (|\nabla u|^{p-2} \nabla u - |\nabla v|^{p-2} \nabla v)\cdot  (\nabla u- \nabla v) d\mu = \int_{\partial \Omega} (f-g) (u-v) d\nu
  \le \|u-v\|_{L^p(\partial\Om)}\|f-g\|_{L^{p'}(\partial\Om)}.
\]
We first consider the case $p\ge 2$. In this case,
we invoke the monotonicity~\eqref{eqn:monotonicity} to obtain
\begin{equation}\label{monotonicity-st}
\| \nabla u- \nabla v \|_{L^p (\Omega)}^p \le C\, \|u-v\|_{L^p(\partial \Omega) } \|f-g\|_{L^{p'}(\partial \Omega)} .
\end{equation}
Now by using the boundedness of the trace operator in Theorem~\ref{thm:Maly-Trace}
and employing Lemma~\ref{lem:control-via-data}, we obtain
\begin{align*}
\Vert u-v\Vert_{L^p(\partial\Om)}\le \Vert u\Vert_{L^p(\partial\Om)}+ \Vert v\Vert_{L^p(\partial\Om)}
  &\le C\left[\Vert u\Vert_{N^{1,p}(\Om)}+ \Vert v\Vert_{N^{1,p}(\Om)}\right]\\
  &\le C\left[\Vert\nabla u\Vert_{L^p(\Om)}+\Vert\nabla v\Vert_{L^p(\Om)}\right]\\
  &\le C\left[\Vert f\Vert_{L^{p'}(\partial\Om)}^{p'/p}+\Vert g\Vert_{L^{p'}(\partial\Om)}^{p'/p}\right].
\end{align*}
Therefore
\[
\|\nabla u-\nabla v\|_{L^p(\Omega)}
  \le C\, \left(\Vert f\Vert_{L^{p'}(\partial\Om)}^{p'/p}+\Vert g\Vert_{L^{p'}(\partial\Om)}^{p'/p}\right)^{1/p}
     \|f-g\|_{L^{p'}(\partial \Omega)}^{1/p},
\]
which yields the desired inequality for the case $p\ge 2$.

On the other hand, if $1<p<2$, then we can proceed as in \eqref{eqn:plessthan2} and invoke H\"older's inequality and~\eqref{eqn:monotonicity} to obtain
\begin{align}\label{monotonicity-pLess2}
\| &\nabla u- \nabla v \|_{L^p (\Omega)}^p\le
\left(\int_\Om(|\nabla u|+|\nabla v|)^p\, d\mu\right)^{(2-p)/2}\left(\int_\Om|\nabla (u-v)|^2 (|\nabla u|+|\nabla v|)^{p-2}\, d\mu\right)^{p/2}
     \notag \\
     & \le C\left(\|\nabla u\|_{L^p(\Om)}^p+\|\nabla v\|_{L^p(\Om)}^p\right)^{(2-p)/2}
         \left(\ \int_\Om (|\nabla u|^{p-2} \nabla u - |\nabla v|^{p-2} \nabla v)\cdot (\nabla u - \nabla v) d\mu
         \right)^{p/2}
     \notag \\
     &\le C\left(\|\nabla u\|_{L^p(\Om)}^p+\|\nabla v\|_{L^p(\Om)}^p\right)^{(2-p)/2}
         \left(\|u-v\|_{L^p(\partial\Omega) } \|f-g\|_{L^{p'}(\partial \Omega)}\right)^{p/2}.
\end{align}
Now, as in the case of $p\ge 2$, we use~Theorem \ref{thm:Maly-Trace}  together with Lemma~\ref{lem:control-via-data}, but this
time also to bound $\|\nabla u\|_{L^p(\Om)}$ and $\|\nabla v\|_{L^p(\Om)}$, to obtain
\[
\| \nabla u- \nabla v \|_{L^p (\Omega)}^p
  \le C\left(\|f\|_{L^{p'}(\partial\Om)}+\|g\|_{L^{p'}(\partial\Om)}\right)^{\tfrac{2-p}{2}p'+\tfrac{p'}{2}}
    \|f-g\|_{L^{p'}(\partial \Omega)}^{p/2}. \qedhere
\]
\end{proof}

\begin{remark}
Theorem \ref{thm:stability-cheeger-intro}, in combination with the Poincar\'e
inequality~\eqref{eq:PI-zeroMean}, yields that if $(f_k)_k$ is a sequence
of functions in $L^{p'}(\partial\Om)$ with the condition that $\int_{\partial\Om}f_k\, d\nu=0$ for each $k$, and a function
$f\in L^{p'}(\partial\Om)$ such that $f_k\to f$ in $L^{p'}(\partial\Om)$, and if $u_k$ is the solution to the Neumann boundary value
problem with boundary data $f_k$ and with $\int_\Om u_k\, d\mu=0$, then $u_k\to u$ in $N^{1,p}(\Om)$ with $u$ the
solution to the Neumann boundary value problem with boundary data $f$ and with $\int_\Om u\, d\mu=0$.
\end{remark}

Next, we turn to  the version of the Neumann problem interpreted as a variational problem involving  upper gradients. In this setting
there is no Euler-Lagrange equation available and so the above argument would not work.
Correspondingly, in this more general setting we obtain a  weaker result,  in the sense that we only can
prove that convex combination of the solutions for data
$f_k$ converge to a solution for the limit data $f_k\to f$. We include this result here to illustrate some of the control we
give up by not having access to the Euler-Lagrange equation provided by differential structure $\nabla u$.

Given $u\in N^{1,p}(\Omega),$ and a $\nu$-measurable function $f:\partial\Omega\to\R$ such that $\int_{\partial\Omega}f\,d\nu=0$,
consider the energy functional
\[
I_f(u):=\int_{\overline{\Omega}}g_u^p\,d\mu-p\int_{\partial\Omega}uf\,d\nu.
\]
Here $g_u$ is the minimal $p$-weak upper gradient of $u$. Let
\[
N^{1,p}_*(\Omega)=\left\{u\in N^{1,p}(\Omega):\int_\Omega u\,d\mu=0\right\},
\]
and define
\[
I_{\min}(f)=\inf_{v\in N^{1,p}_*(\Omega)}I_f(v).
\]

\begin{thm}\label{thm:stability-uppergradient}
For each $k\in\N$, let $f_k\in L^{p'}(\partial\Omega)$ be such that $\int_{\partial\Omega}f_k\,d\nu=0,$ where
$p'$ is the H\"older dual $p/(p-1)$ of $p$.  Let $u_k\in N^{1,p}_*(\Omega)$ be such that $I_{f_k}(u_k)=I_{\min}(f_k)$.
Suppose that $f_k\to f$ in $L^{p'}(\partial\Omega)$. Then
$\int_{\partial\Omega}f\,d\nu=0$, and
there is a convex combination sequence $v_k=\sum_{j=k}^{N(k)}\lambda_{j,k}u_j$ that converges in $N^{1,p}(\Omega)$ to
a function $u\in N^{1,p}_*(\Omega)$ and we have that $I_f(u)=I_{\min}(f)$.
\end{thm}

In the above, note that for each $k,j$ we have $0\le \lambda_{j,k}\le 1$ with $\sum_{j=k}^{N(k)}\lambda_{j,k}=1$.

\begin{proof}
Since $\int_{\partial\Omega}f_k\,d\nu=0$ for all $k\in\N$ and $f_k\to f$ in $L^{p'}(\partial\Omega),$ it is clear that
$\int_{\partial\Omega} f\,d\nu=0$.

For each $k$, we have that $$I_{f_k}(u_k)=I_{\min}(f_k)\le 0,$$ since the zero function belongs to $N^{1,p}_*(\Omega).$
Thus by~\cite[Proposition~4.1]{MS}, there exists $C>0$ such that for all $k\in\N$,
\[
0\ge I_{f_k}(u_k)\ge\int_{\Omega}g_{u_k}^p\,d\mu-C\left(\int_\Omega g_{u_k}^p\,d\mu\right)^{1/p}
  \left(\int_{\partial\Omega}|f_k|^{p'}\,d\nu\right)^{1/p'}.
\]
We then have that
\[
\int_\Omega g_{u_k}^p\,d\mu\le C\int_{\partial\Omega}|f_k|^{p'}\,d\nu\le C_0,
\]
for some $0<C_0<\infty$, since $f_k\to f$ in $L^{p'}(\partial\Omega)$.  Furthermore, since $u_k\in N^{1,p}_*(\Omega),$ it follows from
the Sobolev-type inequality~\cite[(3.6)]{MS}
that there exists some $C>0$ such that for all $k\in\N$,
\[
\|u_k\|_{L^p(\Omega)}\le C\|g_{u_k}\|_{L^p(\Omega)}.
\]
Thus, $(u_k)_{k\in\N}$ is a bounded sequence in $N^{1,p}(\Omega)$, and so for each $k\in\N$, there is a convex combination
\[
v_k:=\sum_{j=k}^{N(k)}\lambda_{j,k}u_j,\quad g_k:=\sum_{j=k}^{N(k)}\lambda_{j,k}g_{u_j}
\]
with $0\le \lambda_{j,k}\le 1$ and
$\sum_{j=k}^{N(k)}\lambda_{j,k}=1$, such that $v_k\to u$ in $L^p(\Omega)$ and $g_k\to g$ to $L^p(\Omega)$ where
$g$ is some $p$-weak upper gradient of $u$, see for instance~\cite[Lemma~3.6]{Sh}
or~\cite[Proposition~7.3.7]{HKST}.  From the
boundedness of the trace $T:N^{1,p}(\Omega)\to L^p(\partial\Omega)$,
\begin{align*}
\left|\int_{\overline\Omega}v_kf\,d\nu-\int_{\partial\Omega}uf\,d\nu\right|
   &\le\int_{\overline\Omega}|(v_k-u)f|d\nu \\
	&\le\|v_k-u\|_{L^p(\partial\Omega)}\|f\|_{L^{p'}(\partial\Omega)}\\
	&\le C\|v_k-u\|_{N^{1,p}(\Omega)}\|f\|_{L^{p'}(\partial\Omega)}\\
	&\le C\left(\|v_k-u\|_{L^p(\Omega)}+\|g_k-g\|_{L^p(\Omega)}\right)\|f\|_{L^{p'}(\partial\Omega)}\to 0
\end{align*}
as $k\to\infty.$  Thus we have that
\begin{align}
I_f(u)=\int_{\overline\Omega}g_u^p\,d\mu-p\int_{\partial\Omega}uf\,d\nu
&\le\int_{\overline\Omega}g^p\,d\mu-p\int_{\partial\Omega}uf\,d\nu\nonumber\\
	&=\lim_{k\to\infty}\left(\int_{\overline\Omega}g_k^p\,d\mu-p\int_{\partial\Omega}v_kf\,d\nu\right).\label{***}
\end{align}

Let $\eps>0.$  Then, there exists $v_0\in N^{1,p}_*(\Omega)$ such that
$$I_f(v_0)<I_{\min}(f)+\eps.$$
We note that for any $v\in N^{1,p}_*(\Omega),$
\begin{equation*}
\left|I_{f_k}(v)-I_{f}(v)\right|
  \le p\int_{\partial\Omega}|v(f_k-f)|\,d\nu\le p\|v\|_{L^p(\partial\Omega)}\|f_k-f\|_{L^{p'}(\partial\Omega)}\to 0
\end{equation*}
as $k\to\infty$, since $f_k\to f$ in $L^q(\partial\Omega)$ and $v\in L^p(\partial\Omega)$ by the boundedness of the trace
operator.  Since $I_{f_k}(u_k)=I_{\min}(f_k)$ for each $k$, we then have that
\begin{equation}\label{**}
\limsup_{k\to\infty}I_{f_k}(u_k)\le\limsup_{k\to\infty}I_{f_k}(v_0)=I_f(v_0)<I_{\min}(f)+\eps.
\end{equation}
This also shows that $\lim_{k\to\infty}I_{\min}(f_k)=I_{\min}(f)$.

By the triangle inequality, we have that
\begin{align*}
\left(\int_{\overline\Omega}g_k^p\,d\mu\right)^{1/p}
        \le\sum_{j=k}^{N(k)}\lambda_{j,k}\left(\int_{\overline\Omega}g_{u_j}^p\,d\mu\right)^{1/p} 
	=\sum_{j=k}^{N(k)}\lambda_{j,k}\left(I_{f_j}(u_j)+p\int_{\partial\Omega}u_jf_j\,d\nu\right)^{1/p}.
\end{align*}
By \eqref{**}, it follows that for sufficiently large $k\in\N,$
\[
\left(\int_{\overline\Omega}g_k^p\,d\mu\right)^{1/p}
  \le\sum_{j=k}^{N(k)}\lambda_{j,k}\left(I_{\min}(f)+\eps+p\int_{\partial\Omega}u_jf_j\,d\nu\right)^{1/p}.
\]
By H\"older's Inequality, we have that
\begin{align*}
\bigg(\int_{\overline\Omega}g_k^p\,d\mu\bigg)^{1/p}
 &\le\left(\sum_{j=k}^{N(k)}\lambda_{j,k}
    \left(I_{\min}(f)+\eps+p\int_{\partial\Omega}u_jf_j\,d\nu\right)\right)^{1/p}\left(\sum_{j=k}^{N(k)}\lambda_{j,k}\right)^{1/p'}\\
	&=\left(\sum_{j=k}^{N(k)}\lambda_{j,k}\left(I_{\min}(f)+\eps+p\int_{\partial\Omega}u_jf_j\,d\nu\right)\right)^{1/p}.
\end{align*}
Therefore it follows that
\[
\int_{\overline\Omega}g_k^p\,d\mu\le I_{\min}(f)+\eps+p\sum_{j=k}^{N(k)}\lambda_{j,k}\int_{\partial\Omega}u_jf_j\,d\nu.
\]
Hence,
\begin{align*}
\int_{\overline\Omega}g_k^p\,d\mu-p\int_{\partial\Omega}v_kf\,d\nu&\le I_{\min}(f)
    +\eps+\sum_{j=k}^{N(k)}\lambda_{j,k}p\int_{\partial\Omega}u_jf_j\,d\nu-p\int_{\partial\Omega}v_kf\,d\nu\\
	&=I_{\min}(f)+\eps+p\sum_{j=k}^{N(k)}\lambda_{j,k}\int_{\partial\Omega}u_j(f_j-f)\,d\nu\\
	&\le I_{\min}(f)+\eps+p\sum_{j=k}^{N(k)}\lambda_{j,k}\|u_j\|_{L^p(\partial\Omega)}\|f_j-f\|_{L^{p'}(\partial\Omega)}\\
	&\le I_{\min}(f)+\eps+C\, p\, \sum_{j=k}^{N(k)}\lambda_{j,k}\|u_j\|_{N^{1,p}(\Omega)}\|f_j-f\|_{L^{p'}(\partial\Omega)}.
\end{align*}
Here we again used the boundedness of the trace operator.  Since the functions
$u_k$ are bounded in $N^{1,p}(\Omega)$ and $f_k\to f$ in $L^{p'}(\partial\Omega)$,
it follows that for sufficiently large $k$,
\[
\int_{\overline\Omega}g_k^p\,d\mu-p\int_{\partial\Omega}v_kf\,d\nu<I_{\min}(f)+2\eps.
\]
Therefore by \eqref{***},  we have that $I_f(u)\le I_{\min}(f)+2\eps$.
By definition, $I_{\min}(f)\le I_{f}(u)$. It follows that $I_{\min}(f)=I_f(u)$; that is, $u$ is a solution to the Neumann boundary value
problem with boundary data $f$.
\end{proof}

\section{Constructing an induced non-local fractional Laplacian for compact doubling metric measure spaces}
\label{sec:construct-fractLap}

\subsection{The general case $1<p<\infty$}

In this section we provide a method  to reproduce the strategy in \cite{CS} and define an analog of fractional $p$-Laplacian operators
$(-\Delta_p)^\theta$ on doubling metric measure
spaces for $1<p<\infty$ and $0<\theta<1$. We recall that in~\cite{EbGKSS} a fractional Laplacian
$(-\Delta_2)^\theta$ corresponding to a Cheeger differential structure
on a complete doubling metric measure space
supporting a $2$-Poincar\'e inequality was constructed using a method different from the one employed in this paper.
In Section~\ref{Sec:Reconcile} below we will show that the definition  in~\cite{EbGKSS}, gives rise to the
same operator we construct in this section.

 While in the Euclidean case \cite{CS}, the $(-\Delta_2)^\theta$ operator  in $\R^n$ arises as the
 Dirichlet-to-Neumann map in  the upper half-space $\R^{n+1}_+$, in our more general setting the operators
 $(-\Delta_p)^\theta$ are defined in terms of the Dirichlet-to-Neumann map in a hyperbolic filling of
 $(Z,d,\nu)$, which satisfies the hypotheses (H0),~(H1), and (H2).

As outlined in Section \ref{sub:hyp-fill}, the hyperbolic filling construction in~\cite{BBS}, shows that any
compact doubling metric measure space $(Z,d,\nu)$, for any fixed
and $1<p<\infty$ with $0<\theta<1$, arises as the boundary of a uniform domain $\Om$,
equipped with a measure $\mu$, such that $(\Om, d,\mu)$ is doubling and supports a $1$-Poincar\'e inequality (and hence, a
$p$-Poincar\'e inequality for each $1<p<\infty$). The original space $Z$ is biLipschitz equivalent to $\partial\Om$,
and the Besov space $B^\theta_{p,p}(Z)$ is the trace space of of the Sobolev class
$N^{1,p}(\Om)$. Indeed, it is shown there that with the same domain $\Om$, for each choice of $1<p<\infty$ and $0<\theta<1$,
there is a choice of measure $\mu$ on $\Om$ satisfying the above properties. Moreover, for that choice of measure $\mu$
we also have from~\cite[Theorem~10.3, Theorem~11.3, Theorem~12.1]{BBS} that when $x\in Z=\partial\Om$ and $0<r<2\diam(Z)$.
\[
\nu(B(x,r))\approx\, \frac{\mu(B(x,r))}{r^\Theta}\, \text{ with }\theta=\frac{p-\Theta}{p},
\]
Hence, \emph{every} compact doubling metric measure space $(Z,d,\nu)$
arises as the boundary of a metric measure
space $(\Om,d,\mu)$ with $\Om$ a uniform domain and $(\Om,d, \mu)$ doubling metric measure space supporting a $p$-Poincar\'e
inequality, with the link between the measure $\mu$ on $\Om$ and the doubling measure $\nu$ on $Z$ given in terms of the
codimensionality condition~\eqref{eq:Co-Dim-intro}.  Since $(\Om,d,\mu)$ satisfies properties~(H0),~(H1), and (H2), we can then fix  a
Cheeger differential structure on $\Om$ as described in Section~2, and apply all the results in the previous sections.
We will use this notation for the rest of the section.

One natural norm on the Besov class $B^\theta_{p,p}(Z)$ corresponds to a form
$\mathcal{E}_p$ given by
\[
\calE_p(u,v)
=\int_{Z}\int_{Z}\frac{|u(y)-u(x)|^{p-2}(u(y)-u(x))(v(y)-v(x))}{d(x,y)^{p\theta}\nu(B(y,d(x,y)))}\, d\nu(x)\, d\nu(y).
\]
Note that when $u,v\in B^\theta_{p,p}(\partial\Omega)$, we have that $\calE_p(u,v)\in\R$ and that
$\calE_p(u,u)=\Vert u\Vert_{\theta,p}^p$. Moreover, $\calE_p(u,u)=0$ if and only if $u$ is constant $\nu$-a.e.~in $\partial\Omega$.
While $\calE_p$ is not a bilinear form nor is symmetric in general, it is both bilinear and symmetric when $p=2$.
However, there are other comparable forms in the Besov class, see for example~\cite{HMY, KRS}. Given the
preceding results of this paper, we have another equivalent form on $B^\theta_{p,p}(Z)$  that is more adapted to
seeing this Besov space as a trace space. For $u,v\in B^\theta_{p,p}(Z)$, we set
\[
\calE_T(u,v):=\int_\Omega|\nabla \widehat{u}|^{p-2}\nabla\widehat{u}\cdot\nabla\widehat{v}\, d\mu,
\]
where $\widehat{u}\in N^{1,p}(\Omega)$ is such that the trace $T\widehat{u}=u$ and $\widehat{u}$ is Cheeger
$p$-harmonic in $\Omega$ (that is, $\widehat{u}$ solves the Dirichlet problem for the Cheeger $p$-Laplacian
on $\Om$ with boundary data $u$). Note that $\calE_T$ is bilinear if and only if $p=2$; otherwise, it is only
linear in the second entry.

\begin{lem}\label{lem:ETvsE}
There exists $C\ge 1$, depending only on the structure constants, such that for each $u\in B^\theta_{p,p}(\partial\Omega)$, we have
\[
\frac{1}{C}\, \calE_T(u,u)\le \calE_p(u,u)\le C\, \calE_T(u,u).
\]
\end{lem}

\begin{proof}
From~\cite[Theorem~1.1]{M},
with $T:N^{1,p}(\Omega)\to B^\theta_{p,p}(\partial\Omega)$ the trace operator and
$E:B^\theta_{p,p}(\partial\Omega)\to N^{1,p}(\Omega)$ the extension operator, we have by the $p$-harmonicity
property of $\widehat{u}$ and by the fact that $T\widehat{u}=TEu$,
\[
\calE_T(u,u)\le \int_\Omega|\nabla Eu|^p\, d\mu\le C\, \Vert u\Vert_{\theta,p}^p=C\, \calE_p(u,u)
\]
and
\[
\calE_p(u,u)=\Vert T\widehat{u}\Vert_{\theta,p}^p\le C\, \int_\Omega|\nabla \widehat{u}|^p\, d\mu=C\, \calE_T(u,u). \qedhere
\]
\end{proof}

We are now ready to construct a Cheeger fractional $p$-Laplacian on $Z$ induced by the Cheeger $p$-Laplacian on $\Om$;
recall that $Z$ is seen as the boundary of $\Om$. This construction is given via the following theorem, and is analogous to the
notion of weak Laplacian $\Delta$.

\begin{prop}\label{thm:solving-fract-laplace}
For each $f\in L^{p'}(Z)$ with $\int_Zf\, d\nu=0$
there is a function $u_f\in B^\theta_{p,p}(Z)$ such that for each $\pip\in B^\theta_{p,p}(Z)$,
\[
\calE_T(u_f,\pip)=\int_Z\pip\, f\, d\nu.
\]
Moreover, there is a constant $C>0$, which depends solely on the structural constants of $\Om$ (or $Z$), such that
for each $f\in L^{p'}(Z)$,
\[
\calE_p(u_f,u_f)\le C\, \int_Z|f|^{p'}\, d\nu.
\]
 If $f\in L^q(Z)$ for sufficiently large $q$, then $u_f$ is H\"older continuous on $Z$.
\end{prop}

\begin{proof}
Given a function $f$ as in the hypothesis of the theorem, let $u_f\in N^{1,p}(\Om)$ be the solution to the Neumann boundary
value problem for the Cheeger $p$-Laplacian on $\Om$, with Neumann boundary data $f$. We will denote the trace of $u_f$
to the boundary $Z$ also by $u_f$.

Note that $u_f$ is Cheeger $p$-harmonic on $\Om$; hence in the construction of $\calE_T$, we have that
$\widehat{u_f}=u_f$, and so for $\pip\in B^\theta_{p,p}(Z)$ we have from Theorem~\ref{thm:equiv-intro} that
\[
\calE_T(u_f,\pip)=\int_\Om |\nabla u_f|^{p-2}\nabla u_f\cdot\nabla \widehat{\pip}\, d\mu=\int_Z\pip\, f\, d\nu.
\]
Moreover, by Lemma~\ref{lem:ETvsE},
\[
\calE_p(u_f,u_f)\le C\, \calE_T(u_f,u_f)=C\, \int_\Om|\nabla u_f|^p\, d\mu.
\]
Combining the above with Lemma~\ref{lem:control-via-data} yields the desired inequality
$\calE_p(u_f,u_f)\le C\, \int_Z|f|^{p'}\, d\nu$.
\end{proof}

\begin{remark}From Theorem~\ref{thm:equiv-intro}(c), we also know that
\[
 |\nabla u|^{p-2}\nabla u\cdot\nabla\eta_\epsilon\,d\mu\rightharpoonup -f\,d\nu.
\]
This behavior corresponds to the behavior of functions $u_f$ as identified in~\cite[page~1247]{CS}
for $p=2$ and $a=0$ (corresponding to $(-\Delta)^{1/2}$), see also~\cite[page~454]{FF} for the Carnot groups setting and~\cite{EbGKSS} for the setting of metric measure spaces with a doubling measure supporting a $2$-Poincar\'e
inequality. From
Theorem~\ref{thm:equiv-intro}(c), with the codimensionality between $\mu$ and $\nu$ given by the exponent $\Theta$,
we have $a=\Theta-1$ in~\cite{CS}. This justifies our definition of fractional $p$-Laplacian in Definition \ref{def:fract-Lap-construct-intro}.
\end{remark}

\begin{proof}[Proof of Theorem \ref{thm:main-fract-Lap-intro}]
We fix $p,\theta$ as in the statement of the theorem. Then, from the results of~\cite{BBS}, we know that there is
a uniform domain $\Om$, equipped with a doubling measure $\mu$ and supporting a $1$-Poincar\'e inequality,
such that $Z$ is biLipschitz equivalent to $\partial\Om$ and with $\nu$, $\mu$ linked via the co-dimensionality
condition~\eqref{eq:Co-Dim-intro} for $\Theta=p(1-\theta)$, as in the discussion in
Subsection~\ref{sub:hyp-fill}. With this choice of $\Om$ we
have the construction of $\calE_T$ as described at the beginning of this section, and the existence of $u_f$ now follows
from Proposition~\ref{thm:solving-fract-laplace}.

The H\"older regularity of $u_f$ follows from Theorem~\ref{thm:holder-intro} upon noting that in gaining H\"older estimates
for points $x,y\in Z$, we consider only the balls $B_i$ centered at points in $Z$, and for such balls we have that
$\mu(B)\approx \left(\rad(B)\right)^{(1-\theta)p}\nu(B)$, and so the relevant lower mass bound exponent for $\mu$ here is
$Q_Z+(1-\theta)p$. Thus, if $q>\max\{1, Q_Z/\theta\}$, then the hypotheses of Theorem~\ref{thm:holder-intro} is satisfied,
and the H\"older regularity of $u_f$ follows.

The $L^{p'}$ stability with respect to the boundary data   $f$ follows from Theorem~\ref{thm:stability-cheeger-intro}.

In order to prove the Harnack inequality,
let $\widehat{u}$ be the solution of the Dirichlet problem for the $p$-Laplacian
in $\Omega$, with boundary data $u$ in $Z$.
We first observe that as $u\ge 0$ on $Z=\partial\Om$, by the strong
maximum principle we have that either $\widehat{u}>0$ in $\Om$ or else $\widehat{u}$ is identically zero in $\Om$
(and hence in $Z$),
see~\cite[Corollary~6.4]{KS}. As the zero function trivially satisfies any Harnack inequality, we only focus on the case
that $u$ is not constant in $\Om$.
Next, we note that $\widehat{u}$
is actually $p$-harmonic in the open set $\Omega \cup W\subset\overline{\Om}$. In fact,
since $(-\Delta_p)^\theta u=f=0$ in $W$, then
for every function $\pip\in N^{1,p}(\Om)=N^{1,p}(\overline{\Om})$ with support contained in $\Om\cup W$, we have
from Theorem~\ref{thm:equiv-intro}(a) that $\int_{\Om\cup W}|\nabla \widehat{u}|^{p-2}\nabla\widehat{u}\cdot\nabla\pip\, d\mu=0$,
which tells us that $\widehat{u}$ is $p$-harmonic in the domain $\Om\cup W$ (seen as a domain in the metric space
$\overline{\Om}$). The normal derivative of $\widehat{u}$ vanishes identically in $W$, and
this allows us to use the results in~\cite{KS}.
Invoking~\cite[Corollary~7.3]{KS}, one has that $\widehat{u}$
satisfies a Harnack inequality on all balls $B$ such that $4B\subset \Om\cup W$.  In particular, for each such ball one has
\[
\sup_{B\cap W} u \le \sup_{B} \hat{u} \le C \inf_B \hat{u} \le C\inf_{B\cap W} u,
\]
concluding the proof.
\end{proof}

\begin{remark}
Bilinear forms such as $\calE_T$ and $\calE_p$, for $p=2$, correspond to a Hunt process or a jump process,
see for example~\cite{BBCK, CKW, FOT}.
\end{remark}

\subsection{The case $p=2$}\label{subsect:p=2Operator}

Much of the extant literature on fractional operators deal with the linear case $p=2$, as in~\cite{BG, BGMN, BGS, CS, CC,CSt, CG, CK1, CK2, CKW, EbGKSS,G}. The fractional Laplacian, $(\Delta)^\theta$ on $Z$, as considered there was studied using
spectral theory, and, in the case of~\cite{BG, BGMN,BGS,CS, CC,CSt, CG, EbGKSS}, was related to the behavior of
the harmonic extension of the solution to a higher-dimensional domain $X$ with the aid of the $2$-Poincar\'e inequality
on the lower-dimensional space $Z$. It was also shown there that the spectral construction of the fractional Laplacian operators
agree with the infinitesimal generator (see for example~\cite{FOT}) of the non-local bilinear form $\calE_2$.
In the current paper we give an intrinsic construction of a fractional Laplacian
on $Z$ by realizing it as the boundary of a John domain, and so in the case $p=2$ we have two approaches to constructing
the fractional Laplacian operators. In the case that $Z$ itself does not support a $2$-Poincar\'e inequality, the spectral
construction as described in~\cite{CS, EbGKSS} is not possible; however, the infinitesimal generator $\mathcal{A}$ of $\calE_2$
and the operator $\mathcal{A}_T$ constructed in the current paper both exist.
For completeness of discussion, we now consider the case $p=2$ and discuss
the construction of the infinitesimal generator $\mathcal{A}$ associated with the bilinear form $\mathcal{E}_2$, as considered
in the above-mentioned literature. To do so,
we need to consider $\mathcal{E}_2$ as a norm; however, it is not a norm on $B^\theta_{2,2}$, and hence we need to extend
the bilinear form to the homogeneous Besov classes; see Subsection~\ref{subsec2.4} for the relevant notions.

In the proposition below, we fix $p=2$, and note by $\mathcal{A}$ the infinitesimal generator associated with the
symmetric bilinear form $\mathcal{E}_2$ on $HB^\theta_{2,2}(\partial\Om)$. For functions $u\in B^\theta_{2,2}(\partial\Om)$,
we say that $u$ is in the domain of $\mathcal{A}$ if there is a function $f\in L^2(\partial\Om)$ such that for each
$v\in B^\theta_{2,2}(\partial\Om)$ we have $\mathcal{E}_2(u,v)=-\int_{\partial\Om} v\, f\, d\nu$. In this case we denote
$\mathcal{A} u=f$.

Since $\partial\Om$ is a bounded set and hence $\nu(\partial\Om)$ is finite, it follows that constant functions are in
$B^\theta_{2,2}(\partial\Om)$. Hence, by using the choice of $v=1$ in the defining identity of $\mathcal{A}u$ in the
above paragraph, we must have $\int_{\partial\Om}f\, d\nu=0$ if $\mathcal{A}u=f$.
Note that if $w,u\in B^\theta_{2,2}(\partial\Om)$ such that $u-w$ is constant on $\partial\Om$, then $u$ is in the domain
of $\mathcal{A}$ if and only if $w$ is; moreover,
$\mathcal{A}u=\mathcal{A}w$.

\begin{prop}
Let $f\in L^2(\partial\Om)$ such that $\int_{\partial\Om}f\, d\nu=0$. Then there exists $w_f\in B^\theta_{2,2}(\partial\Omega)$ such that
$\calA\, w_f=f$.
\end{prop}

\begin{proof}
From~\cite{MS} we know that $u_f$ that satisfies any of the conditions set forth in Theorem~\ref{thm:equiv-intro} exists.
It follows from the above discussion that $\calA_T \, Tu_f=f$ and $\widehat{Tu_f}=u_f$. From the discussion above
we have 
\[
\bigg\vert\int_{\partial\Omega}\pip\, f\, d\nu\bigg\vert\le C_0\, \Vert\pip\Vert_{\theta,2}
=C_0\, \mathcal{E}_2(\pip,\pip)^{1/2},
\]
with $C_0$ depending on $f$.
We know that $\Vert\cdot\Vert_{\theta,2}$ is a norm on the homogeneous space
$HB^\theta_{2,2}(\partial\Omega)=B^\theta_{2,2}(\partial\Omega)/\hskip-.2cm\sim$, and
by above, $\pip\mapsto\int_{\partial\Omega}\pip\, f\, d\nu$ is a bounded linear map on $HB^\theta_{2,2}(\partial\Omega)$
because $\int_{\partial\Omega}\pip\, f\, d\nu=\int_{\partial\Omega}(\pip-c)\, f\, d\nu$ for each real number $c$.
By Lemma~\ref{lem:Lp-BesovE} we have that $HB^\theta_{2,2}(\partial\Omega)$ is a reflexive
Banach space, and so
it follows from the Riesz representation theorem that there is some function (up to a constant)
$w_f\in HB^\theta_{2,2}(\partial\Omega)$ such that for each $\pip\in HB^\theta_{2,2}(\partial\Omega)$,
\[
 \int_{\partial\Omega}\pip\, f\, d\nu=\calE_2(w_f,\pip). \qedhere
\]
\end{proof}

The symmetric non-local bilinear form $\mathcal{E}_2$, as described above, is part of a class of symmetric non-local
bilinear forms studied in~\cite{CKW}. The version of Poincar\'e inequality considered in~\cite[Definition~1.19]{CKW} is
tautological for the form $\mathcal{E}_2$ considered above with $\phi(r)=r^2$, and so by~\cite[Theorem~1.20]{CKW} we
have that an $\mathcal{E}_2$-harmonic function is necessarily H\"older continuous on its domain of harmonicity. Here,
from the discussion in~\cite[Section~2]{CKW}, a function $u$ is $\mathcal{E}_2$-harmonic in an open set $U\subset X$
if $u\in B^\theta_{2,2}(X)$ and for all $\pip\in B^\theta_{2,2}(X)$ with compact support in $U$ we have
$\mathcal{E}_2(u,\pip)=0$. In particular, our construction $w_f$ is $\mathcal{E}_2$-harmonic in an open set $U\subset X$
if $f=0$ on $U$; it then follows from~\cite[Theorem~1.20]{CKW} that $w_f$ is H\"older continuous on $U$. The
results in~\cite{CKW} do not extend to the case where $f$ is not zero on $U$, and so for more general $f$, we do not
know whether $w_f$ is H\"older continuous, but from Theorem~\ref{thm:holder-intro} above we know that $u_f$ is
indeed H\"older continuous when $f\in L^q(X)=L^q(\partial\Om)$ for $q$ sufficiently large.

\section{Reconciling construction of fractional Laplacian with~\cite{EbGKSS}} \label{Sec:Reconcile}

In Section~\ref{sec:construct-fractLap} we gave a possible construction of a fractional $p$-Laplacian operator
$(-\Delta_2)^\theta$ on a doubling metric measure space $\partial\Omega$.
In the special case $p=2$, and with the additional hypothesis  that $\partial\Omega$ also supports a $2$-Poincar\'e inequality, an alternative construction based on spectral theory can be found in~\cite{EbGKSS}, which corresponds to the operator
$\mathcal{A}$ described in Subsection~\ref{subsect:p=2Operator}.
The aim of this section is to reconcile these two different approaches, and  show that the construction given in~\cite{EbGKSS}
gives rise to the same fractional operator we define in
Section~\ref{sec:construct-fractLap} above.

In~\cite{EbGKSS} the object of study was a compact doubling metric measure space $(Z,d_Z,\mu)$ that supports a $2$-Poincar\'e
inequality, and this metric space is naturally seen as the boundary of the unbounded domain
$Z\times(0,\infty)$, where $X=Z\times[0,\infty)$ is equipped with the $\ell_2$-product metric $d$.
Whereas, in our paper we consider the boundary of a \emph{bounded} domain. Therefore, to show that the two approaches
are not contradictory, we show that we can modify $X$ so that it becomes a bounded doubling metric measure space
supporting a $2$-Poincar\'e inequality and that functions that are $2$-harmonic in $Z\times(0,\infty)$ or $X$ are also
$2$-harmonic in this modified space, and that $Z$ is isometric to the boundary of this modified space.
We consider the metric $d_\infty$ on $X$ given by
$d_\infty((x_1,y_1),(x_2,y_2))=\max\{d(x_1,x_2),|y_1-y_2|\}$.

Fix $\beta>1$,
and set $\rho$ and $\omega$ to be the following continuous functions on $X$:
\[
\rho(x,y):=\min\{1, \, y^{-\beta}\}, \qquad \omega(x,y)=\min\{1,\, y^{-2\beta}\}.
\]
The metric $d_\rho$ on $X$ is given by
\[
d_\rho((x_1,y_1),(x_2,y_2))=\inf_\gamma\int_\gamma \rho\, ds,
\]
where the infimum is over all rectifiable curves in $X$ with end points $(x_1,y_1), (x_2,y_2)$. Since $\rho$ is a positive continuous
function and $X$ is complete, it follows that the topology generated by $d_\rho$ is the same as the topology generated by the
original metric on $X$. Moreover, for $(x_0,y_0)\in X$,
\[
\lim_{(x,y)\to (x_0,y_0)}\frac{d_\rho((x,y),(x_0,y_0))}{d((x,y),(x_0,y_0))}=\rho(x_0,y_0),
\]
and hence $(X,d_\rho)$ is a geodesic space.
The metric $d_\rho$ is motivated by the procedure of sphericalization as constructed in~\cite{BB}, see also~\cite{DL1,DL2,L,LS}.
However, the measure also needs to be modified, not as in~\cite{DL1,DL2,L,LS}, but in the manner of~\cite{BBL}.
The modified measure $\mu_\omega$ is given by
\[
\mu_\omega(A)=\int_A \omega\, d\mu_X
\]
where $\mu_X$ is the product measure on $X$ given by $d\mu_X(x,y)=y^a\, d\mu(x)\, dy$, with $a=1-2\theta$ as in~\cite{EbGKSS}.
From the construction, it is clear that the completion $\widehat{X}$ of $X$ with respect to the metric $d_\rho$ is compact.

We denote the arc-length measure on a curve $\gamma$ in $(X,d_\infty)$ by $ds$; then
under the deformed metric $d_\rho$, the arc-length measure $ds_\rho$ is given by $ds_\rho=\rho \, ds$. It follows
that if $g$ is an upper gradient of a function $u$ on $X$ or $Z\times(0,\infty)$, then $\rho^{-1} g$ is an upper gradient of $u$
on $(X,d_\rho)$ or $(Z\times(0,\infty),d_\rho)$. This is because
\[
\int_\gamma g_0\, ds_\rho=\int_\gamma g_0\, \rho\, ds
\]
whenever $g_0$ is a Borel function on $X$ or $Z\times(0,\infty)$. Observe that
\[
\int_X \left(\frac{g}{\rho}\right)^2\, d\mu_\omega=\int_X g^2\, d\mu_X,
\]
and so a family of curves $\Gamma$ is of zero $2$-modulus in $(X,d_\infty)$ if and only if it is of zero $2$-modulus
in $(X,d_\rho)$. Moreover, if $g_u$ is a minimal $2$-weak upper gradient of $u$ in $(X,d_\infty)$ and
$g_{u,\rho}$ is a minimal $2$-weak upper gradient of $u$ in $(X,d_\rho)$, then $g_{u,\rho}=\rho^{-1}\, g_u$,
with
\[
\int_X g_{u,\rho}^2\, d\mu_\omega=\int_X g_u^2\, d\mu_X.
\]
This means that the upper gradient energy is the same with respect to both metrics.

Note that $(X,d_\rho)$ is a bounded metric space, for
\[
\frac12\diam_\rho(X)\le \int_0^\infty \rho(x_0,y)\, dy=1+\int_1^\infty y^{-\beta}\, dy=1+\frac{1}{\beta-1}=\frac{\beta}{\beta-1}<\infty.
\]
On the other hand, $(X,d_\rho)$ is not complete; we complete it by including the ``point at infinity". So we set
$\widehat{X}=X\cup\{\infty\}$, with $d_\rho((x,y),\infty):=\int_y^\infty\rho(x,y)\, dy$, that is,
\begin{equation}\label{e:dist-infty}
d_\rho((x,y),\infty)=\begin{cases} \frac{1}{(\beta-1)\, y^{\beta-1}} &\text{ if }y\ge 1,\\
   \frac{\beta}{\beta-1}-y &\text{ if }0\le y\le 1. \end{cases}
\end{equation}
We wish to consider as the domain $\Omega$ the set $\widehat{X}\setminus (Z\times\{0\})$, and note that
the restriction of $d_\rho$ to $\partial\Omega=Z\times\{0\}$ is isometric to $d_Z$.
Observe that if $x_0\in Z$ and $(x_1,y_1)\in\widehat{X}$ such that
$d_\rho((x_0,0),(x_1,y_1))\le 1$,
then the $\rho$-geodesic connecting $(x_0,0)$ to $(x_1,y_1)$ must lie in the region $Z\times[0,1]$, and so
$d_\rho((x_0,0),(x_1,y_1))=d((x_0,0),(x_1,y_1))=\max\{d_X(x_1,x_0),|y_1|\}$. Hence for each $x_0\in Z$ and
$r\le 1$, we have
\[
\mu_\omega(B_\rho((x_0,0),r))=(1+a)^{-1}\ r^{1+a}\, \mu_Z(B(x_0,r)).
\]
Thus, for radii $r<1$, balls centered at
$(x_0,0)$ are of co-dimension $1+a$ measure with respect to $\mu_X$. Note that then we have $\Theta=1+a$ and
$\theta=\tfrac{1-a}{2}$, as required in the current note.

We now need to know what the
effect of the inclusion of the point at infinity has on the class of $2$-weak upper gradients.

\begin{lem}\label{lem:infty-point}
Let $\Gamma$ be the collection of all non-constant rectifiable curves in $(\widehat{X},d_\rho)$ that passes through
$\infty$. Then $\Mod_2(\Gamma)=0$.
\end{lem}

\begin{proof}
It suffices to show that for each fixed $L\ge 1$, the collection $\Gamma_L$ of all curves with one end point in $Z\times\{L\}$ and the
other at $\infty$ satisfies $\Mod_2(\Gamma_L)=0$, for $\Gamma=\bigcup_{n\in\N}\Gamma_n$. For each $H>L$, we set
$\rho_H=(\beta-1)H^{\beta-1}\chi_{Z\times[H,\infty)}$. A direct calculation shows that $\rho_H$ is admissible for
computing $\Mod_2(\Gamma_L)$ as every curve in $\Gamma_L$ has a subcurve in $Z\times[H,\infty)$ with one
end point at $Z\times\{H\}$ and the other at $\infty$. Note that as $a<1$, we have $\tfrac{a+1}{2}<1<\beta$.
Therefore
\[
\Mod_2(\Gamma_L)\le \int_Z\int_H^\infty \rho_H^2\, d\mu_\omega=\mu(Z)\, (\beta-1)^2H^{2\beta-2}\int_H^\infty y^{a-2\beta}\, dy
   =\frac{(\beta-1)^2}{2\beta-a-1}\, \frac{1}{H^{1-a}}.
\]
Letting $H\to\infty$ and noting that $a<1$ yields the desired conclusion.
\end{proof}

From the above lemma, it is clear that a $2$-weak upper gradient of a function $u$ on $Z\times(0,\infty)$ (or $X$)
extends as a $2$-weak upper gradient of $u$ on $\Omega$ (or $\widehat{X}$). Moreover, the total Newton-Sobolev
$2$-capacity of $\{\infty\}$ is zero.

\subsection{Doubling property of $\mu_\omega$:}

We want to show the existence of a constant $C\ge 1$ such that whenever $0<r<1/4$,
$\mu_\omega(B_\rho((x_0,y_0),2r))\le C\mu_\omega(B_\rho((x_0,y_0),r))$. We first need the following lemma.

\begin{lem}\label{lem:Harnack-prop-w}
There exists $C_\omega>0$ such that whenever $B$ is a ball (with respect to the metric $d_\rho$)
in $\widehat{X}$ with radius $R>0$ and center $(x_0,y_0)$
such that $\infty\not\in 2B$, then for each $(x,y)\in B$ we have
\[
\frac{\omega(x_0,y_0)}{C_\omega}\le \omega(x,y)\le C_\omega\, \omega(x_0,y_0),
\qquad  \frac{\rho(x_0,y_0)}{C_\omega}\le \rho(x,y)\le C_\omega\, \rho(x_0,y_0).
\]
Moreover, for each $(x,y)\in B$ we have
\[
\frac{\rho(x_0,y_0)}{C_\omega}\, d((x_0,y_0),(x,y))\le d_\rho((x_0,y_0),(x,y))\le C_\omega\, \rho(x_0,y_0)\, d((x_0,y_0),(x,y)),
\]
and if in addition we have $y_0\ge 1$ and $y>1$, then
\begin{align}\label{eq:dRoh-d}
 \frac{1}{C_\omega(\beta-1)}\bigg\vert \frac{1}{y_0^{\beta-1}}-\frac{1}{y^{\beta-1}}\bigg\vert +\frac{1}{C_\omega\, y_0^\beta}d_Z(x,x_0)
&\le d_\rho((x,y),(x_0,y_0))\notag\\
&\le
 \frac{C_\omega}{\beta-1}\bigg\vert \frac{1}{y_0^{\beta-1}}-\frac{1}{y^{\beta-1}}\bigg\vert +\frac{C_\omega}{y_0^\beta}d_Z(x,x_0)
\end{align}
\end{lem}

\begin{proof}
If $y_0\le 1$, then $\omega(x_0,y_0)=1$.
If for all $(x,y)\in B$ we have that $y\le 1$, then $\omega(x,y)=1=\omega(x_0,y_0)$. Hence, without loss of generality, we may
assume that there is some $(x,y)\in B$ with $y>1$.
As $\infty\notin 2B$, we have
\[
2R\leq \int_{y_0}^\infty\rho dy = 1-y_0+\int_1^\infty y^{-\beta}dy=1-y_0+\frac{1}{\beta-1}.
\]
As $(x,y)\in B_R$,
\[
R>\int_{y_0}^y \rho dy=1-y_0+\frac{1}{\beta-1}(1-y^{1-\beta}).
\]
By combining the previous two estimates we get
\[
\frac{1-y_0}{2}+\frac{1}{2(\beta-1)}(1-2y^{1-\beta})<0.
\]
This can only be satisfied if the second term is negative. Thus we have to have
\[
y<2^{1/(\beta-1)}
\]
and so $\omega(x,y)\approx 1\approx\rho(x_0,y_0)$, satisfying the first claim of the lemma.

Now we consider the case $y_0>1$. As $\infty \notin 2B$, we obtain
\[
2R<\int_{y_0}^\infty\rho dy \leq \int_{y_0}^\infty y^{-\beta}dy=\frac{1}{\beta-1}y_0^{1-\beta}.
\]
If $y\geq 1$, then $(x,y)\in B$ implies that
\[
\left| \int_{y_0}^y y^{-\beta}dy\right|=\frac{1}{\beta-1}\left| y_0^{1-\beta}- y^{1-\beta} \right|<R.
\]
A combination of the above two estimates gives us that
\[
\frac12 y_0^{1-\beta}< y^{1-\beta} <\frac32 y_0^{1-\beta}.
\]
It follows that $y\approx y_0$, and so again $\omega(x,y)\approx \omega(x_0,y_0)$ and
$\rho(x,y)\approx \rho(x_0,y_0)$.
If $y\leq 1$, then as $(x_0,y_0)\in B((x,y),R)$, it follows by the discussion in the first paragraph above
that we have $\omega(x,y)=1\approx \omega(x_0,y_0)$. This completes the proof of the first claim.

Now we let $(x,y)\in B$, and note that any $d_\rho$-geodesic connecting
$(x_0,y_0)$ to $(x,y)$ lies in $B$; hence by the above, we have that
\[
\frac{1}{C_\omega} \rho(x_0,y_0)\, d((x,y), (x_0,y_0)) \le d_\rho((x,y),(x_0,y_0)).
\]
On the other hand, let $\gamma_Z$ be a geodesic curve in $Z$ with end points $x,x_0$ and parametrized to be from
the interval $[0,|y-y_0|]$ and with constant speed,
and $\gamma_\R$ be the vertical
line segment $\gamma_\R(t)=t+\min\{y,y_0\}$, for $0\le t\le |y_0-y|$. Let $\gamma$ be the curve given by
$\gamma(t)=(\gamma_Z(t),\gamma_\R(t)$
Note by the discussion above, we have that either $y_0\le 1$,
in which case $0\le y\le 2^{1/(\beta-1)}$ and so for each $t$ we have that $\omega(\gamma(t))\approx\omega(x_0,y_0)$,
or $y_0>1$, in which case $y\approx y_0$ and so again $\rho(\gamma(t))\approx\rho(x_0,y_0)$. Hence
\[
d_\rho((x,y),(x_0,y_0))\le \ell_\rho(\gamma)\le C\rho(x_0,y_0) \, d((x,y), (x_0,y_0)).
\]
Thus the second claim of the lemma is also verified.

To verify the last claim, we now suppose that $y\ge 1$ and $y_0\ge 1$.
Note that $\widehat{X}$ is a geodesic space, and hence every geodesic (with respect to the metric $d_\rho$)
connecting the center
$(x_0,y_0)$ of $B$ to $(x,y)\in B$ lies entirely in $B$. Let $\gamma:[0,L]\to X$ be such a geodesic, arc-length
parametrized with respect to the original metric $d$ on $X$. Then
with $\gamma(t)=(\gamma_Z(t),\gamma_{\R}(t))$, and as both $y_0,y\ge 1$,
\begin{align*}
d_\rho((x_0,y_0),(x,y))
=\int_0^L \gamma_{\R}(t)^{-\beta}\, \max\{|\gamma_Z^\prime|(t),|\gamma_{\R}^\prime|(t)\}\, dt
&\ge \int_0^L\gamma_{\R}(t)^{-\beta}|\gamma_{\R}^\prime|(t)\, dt\\
&\ge \frac{1}{\beta-1}\bigg\vert \frac{1}{y_0^{\beta-1}}-\frac{1}{y^{\beta-1}}\bigg\vert.
\end{align*}
Moreover, for each $t\in[0,L]$ we have that $\gamma_{\R}(t)\approx y_0$ by the argument above related to the first claim
of the lemma. It follows that
\[
d_\rho((x_0,y_0),(x,y))\ge \frac{1}{C_\omega\, y_0^\beta}\int_0^L |\gamma_Z^\prime|(t)\, dt
  \ge \frac{1}{C_\omega\, y_0^\beta} d_Z(x,x_0).
\]
Combining this with the above inequality yields the first of the two inequalities in~\eqref{eq:dRoh-d}. The right-hand side
of~\eqref{eq:dRoh-d} is obtained by considering the $\rho$-length of the curve $\gamma_0$ obtained by concatenating
the curve $\gamma_{0,\R}:[0,|y-y_0|]\to X$ given by $\gamma_{0,\R}(t)=(x_0,y_0+\tfrac{y-y_0}{|y-y_0|}t)$
with a geodesic $\gamma_{0,Z}(t)=(\gamma_{1,Z}(t),y)$ where $\gamma_{1,Z}$ is any geodesic in $Z$
with end points $x_0$, $x$.
\end{proof}

\begin{lem}
The measure $\mu_\omega$ is doubling on $\widehat{X}$.
\end{lem}

In the proof below, each occurrence of $C_{\beta,a}$ denotes a possibly different constant, whose choice depends solely on
the parameters $\beta$ and $a$.

\begin{proof}
Let $B_\rho$ be a ball of radius $r>0$ in $\widehat{X}$.  Since $\widehat{X}$ is compact,
it suffices to prove the doubling property for balls of radius $r$ small enough, namely,
\begin{equation}\label{eq:r-limit}
0<r\le \min\bigg\lbrace1/4,\, \frac{1}{(8\diam(Z))^{\beta-1}(4(\beta-1))^\beta}\,,\, \frac{1}{8(\beta-1)}\bigg\rbrace.
\end{equation}
We consider three cases.

\noindent {\bf Case 1:} $\infty\in \tfrac12B_\rho$. In this case, note that $B_\rho(\infty,r/2)\subset B_\rho$ and
$2B_\rho\subset B_\rho(\infty, 3r)$. Note that when $0<R<1/(\beta-1)$, $B_\rho(\infty, R)=Z\times[H_R,\infty)$ with
\begin{equation}\label{eq:height-inftyBall}
H_R=\left(\frac{1}{(\beta-1)R}\right)^{1/(\beta-1)}.
\end{equation}
It follows that
\begin{align}\label{eq:ball-at-infty}
\mu_\omega(B_\rho(\infty,R))=\mu(Z)\, \int_{H_R}^\infty y^{a-2\beta}\, dy&=\frac{\mu(Z)}{2\beta-1-a}\, H_R^{1+a-2\beta}\notag\\
   &=\mu(Z)\, \frac{(\beta-1)^{\tfrac{2\beta-1-a}{\beta-1}}}{2\beta-1-a}\, R^{\tfrac{2\beta-1-a}{\beta-1}}.
\end{align}
Therefore,
\begin{align*}
\mu_\omega(2B_\rho)\le \mu_\omega(B_\rho(\infty,3r))=C_{\beta,a}\, \mu(Z)\, r^{\tfrac{2\beta-a-1}{\beta-1}}
 &= C_{\beta, a}\,
 \mu_\omega(B_\rho(\infty,r/2))\\
 &\le C_{\beta,a}\,
 \mu_\omega(B_\rho),
\end{align*}
that is, the doubling property holds for balls falling within this case, with doubling constant that depends only on $\beta$ and $a$.

\noindent {\bf Case 2:} $\infty\in 4B_\rho\setminus \tfrac12B_\rho$. Since we assume that $r$ is small enough
(see~\eqref{eq:r-limit}), we have that necessarily $y_0>1$. Since $r/4\le d_\rho((x_0,y_0),\infty)<4r$,
by~\eqref{e:dist-infty} it follows that
\begin{equation}\label{eq:y0-r}
\frac{r}{2}\le \frac{1}{(\beta-1)y_0^{\beta-1}}<4r.
\end{equation}
We fix a positive real number $\Delta$ (independently of $r$) such that
\[
\Delta<\min\bigg\lbrace 1,\, \left(\frac{16}{15}\right)^{\beta/(\beta-1)}-1\, ,\, 1-\left(\frac{16}{17}\right)^{1/(\beta-1)}\bigg\rbrace.
\]
Then by~\eqref{eq:r-limit} and~\eqref{eq:y0-r},
we see that $X\times [y_0(1-\Delta),y_0(1+\Delta)]\subset B_\rho$. A direct computation shows then
that
\begin{align*}
\mu_\omega(B_\rho)&\ge \mu_\omega(X\times [y_0(1-\Delta),y_0(1+\delta)])\\
  &=\frac{\mu(Z)}{(2\beta-1-a)y_0^{2\beta-1-a}}\left[(1-\Delta)^{1+a-2\beta}-(1+\Delta)^{1+a-2\beta}\right]\\
   &\approx r^{(2\beta-1-a)/(\beta-1)}.
\end{align*}
The comparison constant in the above depends only on the choice of $\Delta$ and on $\mu(Z)$, $\beta$.
As $\infty\in 4B_\rho$, we see that $4B_\rho\subset B_\rho(\infty, 8r)$. From the computations as in Case~1 above tells us that
\[
\mu_\omega(2B_\rho)\le \mu_\omega(B_\rho(\infty, 8r))=\mu(Z)\, C_{\beta,a}\
r^{(2\beta-1-a)/(\beta-1)}.
\]
Combining this estimate with the antecedent estimate gives $\mu_\omega(2B_\rho)\lesssim \mu_\omega(B_\rho)$.

\noindent {\bf Case 3:} $\infty\not\in 4B_\rho$. In this case, by the first part of
Lemma~\ref{lem:Harnack-prop-w} we know that
$\mu_\omega(2B_\rho)\approx \omega(x_0,y_0)\mu(2B_\rho)$ and $\mu_\omega(B_\rho)\approx\omega(x_0,y_0)\mu(B_\rho)$.
Therefore it suffices to show that $\mu(2B_\rho)\le C\, \mu(B_\rho)$ for some constant $C$ that does not depend on $B_\rho$.
Now by the last part of the lemma, we know that for $(x,y)\in 2B_\rho$ we have that
\[
d_\rho((x_0,y_0),(x,y))\approx y_0^{\beta-1}d((x,y),(x_0,y_0)).
\]
Thus $2B_\rho\subset B((x_0,y_0),2Cry_0^\beta)$
and $B((x_0,y_0),ry_0^\beta/C)\subset B_\rho$. Now the doubling property of $\mu$ yields the desired inequality, with the doubling
constant depending only on the doubling constant of $\mu$ and the constant $C$ used here.
\end{proof}

\subsection{Modification of the norm on Cheeger differential structure}\label{subsect-Cheeger-mod}

The paper~\cite{EbGKSS} considers $D$ to be a (fixed choice of) Cheeger differential structure on $Z$. The existence of such a
structure is guaranteed from the results in~\cite{Che}. The corresponding differential structure $\nabla$ on $X$ is obtained as a
tensorization of the structure $D$ with the Euclidean differential structure on $[0,\infty)$. When the metric on $X$ is modified
as considered in this section, the norm associated with the differential structure on $X$ is also changed. Note that the norm, as
considered in Cheeger in~\cite{Che}, has the property that
\[
|\nabla u(x,y)|_{(x,y)}\approx g_u(x,y)
\]
for $\mu\times\mathcal{L}^1$-a.e.~$(x,y)\in X$; here, $g_u$ is a $2$-weak upper gradient of $u$ with respect to the
original (unmodified) metric. Since $\rho^{-1}g_u$ is a $2$-weak upper gradient of $u$ with respect to the metric $d_\rho$ and
the measure $\mu_\omega$, it follows that the norm on the inner product structure $\nabla$ should change accordingly in order
to preserve the above comparison. Hence we choose
\[
|\nabla u(x,y)|_{\rho, (x,y)}=\rho(x,y)^{-1}|\nabla u(x,y)|_{(x,y)}.
\]
With this modification, we have that for each Borel $A\subset \widehat{X}$,
\[
\int_A \langle \nabla u(x,y), \nabla v(x,y)\rangle_{\rho,(x,y)} \mu_\omega(x,y)
=\int_A\langle \nabla u(x,y), \nabla v(x,y)\rangle_{(x,y)}\, d(\mu\times\mathcal{L}^1)(x,y).
\]

\subsection{$2$-Poincar\'e inequality with respect to $d_\rho$ and $\mu_\omega$}

It was shown in~\cite[Proof of Lemma~3.1]{EbGKSS} that as $Z$ supports a $2$-Poincar\'e inequality, the
product space $X=Z\times(0,\infty)$
also supports a $2$-Poincar\'e inequality. In this subsection we will show that the modified space $\widehat{X}$ also
supports a $2$-Poincar\'e inequality.

\begin{prop}
The space $(\widehat X, d_{\rho},\mu_{\omega})$ supports a $(1,2)$-Poincar\'e inequality.
\end{prop}

\begin{proof}
As the metric $d_{\rho}$ is geodesic, a weak Poincar\'e inequality implies the strong Poincar\'e inequality.
Thus it is enough to prove that for every $u\in D^{1,2}(\widehat X)$ and every ball $B\subset \widehat X$,
there exists a constant $c_{u,B}$ such that
\[
\jint_{B_\rho(x,r)}|u-c_{u,B}|\, d\mu_\omega\le C_P\, r\, \left(\jint_{B_\rho(x,\lambda r)}g_u^p\, d\mu_\omega\right)^{1/2},
\]
with some constants $C_{P}$ and $\lambda$ that only depend on the data of the space.
We consider separately two cases: The case, where the point $\infty$ is far away from $B$, and the case where it is close to $B$.

{\bf Case 1: $\infty\notin B_{\rho}((x_0,y_{0}),2C_\omega^2 r_0)$}, where $C_\omega$ is as in Lemma~\ref{lem:Harnack-prop-w}.
Then, with $\widetilde B=B((x_{0},y_{0}), C_\omega\rho(x_{0},y_{0})^{-1}r_{0})$,
by Lemma~\ref{lem:Harnack-prop-w} we have
\[
B_{\rho}=B_{\rho}((x_{0},y_{0}),r_{0})\subset \widetilde B\subset C_{\omega}^2B_{\rho}= B_{\rho}((x_{0},y_{0}),C_\omega^2r_{0}).
\]
Now, setting $c_{u,B}=\mu_Z(\widetilde B)^{-1}\int_{\widetilde{B}} u\, d\mu_X$,
we can estimate using Lemma~\ref{lem:Harnack-prop-w} and the Poincar\'e inequality of $(X,d_X,\mu_X)$ that
\[
\begin{split}
\int_{B_{\rho}}|u-c_{u,B}|d\mu_{\omega}\,
& \leq C_\omega\, \omega(x_{0},y_{0}) \int_{\widetilde B}|u-c_{u,B}|d\mu_{X} \,\\
&\leq C_\omega^2\, \omega(x_{0},y_{0}) \rho(x_{0},y_{0})^{-1}r_{0}\, \mu_{X}(\widetilde B)^{1/2}
\left( \int_{\widetilde B}g_{u}^2d\mu_{X}   \right)^{1/2}\\
&= C_\omega^2\,  \omega(x_{0},y_{0}) \rho(x_{0},y_{0})^{-1}r_{0} \mu_{X}(\widetilde B)^{1/2}
\left( \int_{\widetilde B}g_{u,\rho}^2d\mu_{\omega}   \right)^{1/2}\\
&\leq C_\omega^2\, r_{0}\, \mu_{\omega}(C_\omega^2B_{\rho})^{1/2}
\left( \int_{C_\omega^2B_{\rho}}g_{u,\rho}^2d\mu_{\omega}   \right)^{1/2}\\
\end{split}
\]
Here the first inequality follows from the inclusion $B_{\rho}\subset \widetilde B$ and from having
$\omega\approx\omega(x_{0},y_{0})$ in $\widetilde B$. The second inequality follows from the
$(1,2)$-Poincar\'e inequality for $(X,d,\mu_{X})$. In the last estimate, we used the facts that for $(x,y)\in B_\rho$,
$\omega(x,y)^{1/2}\approx\omega(x_{0},y_{0})^{1/2}= \omega(x_{0},y_{0})\rho(x_{0},y_{0})^{-1}$ in $\widetilde B$
and that $\widetilde B\subset C^2B_{\rho}$. Thus the Poincar\'e inequality is satisfied with
$\lambda = C_\omega^2$ and $C_{P}$ that depends on $C_\omega$ as well as the doubling constant of $\widehat X$
and the constants associated with the Poincar\'e inequality for $(X,d_X,\mu_X)$.

{\bf Case 2: $\infty\in B_{\rho}((x_0,y_{0}),2C_\omega^2 r_0)$}.

Now we have
$B_{\rho}((x_{0},y_{0}),r_{0})\subset B_{\rho}(\infty,(2C_\omega^2+1)r_{0})=Z\times(H_{(2C_\omega^2+1)r_{0}},\infty]$,
where $H_R=[(\beta-1)R]^{-1/(\beta-1)}$ as in~\eqref{eq:height-inftyBall}.
Therefore as we already know that $\mu_{\omega}$ is doubling, it is enough to prove that the balls
centered at $\infty$ satisfy the Poincar\'e inequality. Thus let us consider the ball
$B_{\rho}=B_{\rho}(\infty,R)$ with $R<1/(2\beta-2)$.

For $k=1,2,\ldots,$ let $A_{k}=B_{\rho}(\infty,2^{1-k}R)\setminus B_{\rho}(\infty,2^{-k}R)$. Then
from equation~\eqref{eq:ball-at-infty} above, we have
$\mu_{\omega}(A_{k+1})/\mu_{\omega}(A_{k})=\gamma=(\tfrac12)^{1+(\beta-a)/(\beta-1)}<1$.
We may assume without loss of generality that $u_{A_{1}}=0$ (with respect to the measure
$\mu_{\omega}$). Then our goal is to prove that
\[
\int_{B_{\rho}}|u|d\mu_{\omega}\leq C R \left(\int_{B_{\rho}}g_{u,\rho}^2d\mu_{\omega}\right)^{1/2}\mu_{\omega}(B_{\rho})^{1/2}.
\]
Note from the discussion around~\eqref{eq:height-inftyBall} that
$A_1=Z\times[H_R, H_{R/2}]$ with $H_R>2$, and so for each $(x,y)\in A_1$ we have that
\[
2^{-2\beta/(\beta-1)}H_R^{-2\beta}=H_{R/2}^{-2\beta}\le \omega(x,y)\le H_R^{-2\beta}.
\]
Note that $H_{R/2}-H_R=(2^{1/(\beta-1)}-1)H_R\ge \diam(Z)$.
Since $(Z,d_Z,\mu_Z)$ supports a $2$-Poincar\'e inequality and so does the Euclidean interval $[H_R,H_{R/2}]$
(equipped with the weighted $1$-dimensional Lebesgue measure $y^a\, dy$), we see that
\begin{equation}\label{eq:A1-est}
\int_{A_1}|u|\, d\mu_X=\int_{A_1}|u-u_{A_1}|\, d\mu_X\le C\, H_R\, \mu_X(A_1)^{1/2}\, \left(\int_{A_1}g_u^2\, d\mu_X\right)^{1/2}.
\end{equation}
Here we used the fact that $H_R>2>\diam(X)$, and so the $d_X$-diameter of $A_1$ is comparable to $H_R$.
Now given the above comparison $\omega((x,y))\approx H_R^{-2\beta}$, we obtain
\begin{align*}
\int_{A_1}|u|\, d\mu_\omega
&\le C H_R^{1-2\beta}\, \left(\int_{A_1}g_{u,\rho}^2\, d\mu_\omega\right)^{1/2} H_R^\beta\, \mu_\omega(A_1)^{1/2}\\
 & = C(\beta-1)\, R\, \left(\int_{A_1}g_{u,\rho}^2\, d\mu_\omega\right)^{1/2}\, \mu_\omega(A_1)^{1/2}.
\end{align*}
Now it is enough to prove that
$
\int_{B_{\rho}\setminus A_{1}} |u|d\mu_{\omega}\leq CR\int_{B_{\rho}}g_{u,\rho}d\mu_{\omega},
$
from which we can recover $(1,2)$-Poincar\'e inequality by using H\"older's inequality.

 As $\omega$ and $\rho$ are approximately constant in each $A_{k}\cup A_{k+1}$ and the $y$-directional width
 with respect to the metric $d_\rho$
 of $A_{k}\cup A_{k+1}$ is smaller than $2^{2-k}R$,
 the use of the $1$-dimensional
 Poincar\'e inequality on $\{x\}\times(0,\infty)$ for fixed $x\in X$ gives us
 \begin{align*}
|u_{A_{k+1}}-u_{A_{k}}|=\bigg\vert \jint_{A_{k+1}}u\, d\mu_\omega-\jint_{A_k}u\, d\mu_\omega\bigg\vert
 &\approx\bigg\vert \jint_{A_{k+1}}u\, d\mu_Z-\jint_{A_k}u\, d\mu_X\bigg\vert\\
 &\le C\, H_{R_k}\, \jint_{A_k\cup A_{k+1}}g_u\, d\mu_X\\
 &\approx C\, H_{R_k}\, \jint_{A_k\cup A_{k+1}}g_u\, d\mu_\omega\\
 &\approx C\, H_{R_k}^{1-\beta} \jint_{A_k\cup A_{k+1}}g_{u,\rho}\, d\mu_\omega\\
 &\approx C\, 2^{-k}\, R\, \jint_{A_k\cup A_{k+1}}g_{u,\rho}\, d\mu_\omega.
 \end{align*}
Here $u_{A_k}$, $k\in\N$, are computed with respect to the measure $\mu_\omega$.
Now we can apply a telescoping argument and~\eqref{eq:A1-est} to estimate that
\[
\begin{split}
\int_{B_{\rho}\setminus A_{1}} & |u|d\mu_{\omega}=  \sum_{k=2}^\infty \int_{A_{k}}|u|d\mu_{\omega}\\
\leq & \sum_{k=2}^\infty  \mu_{\omega}(A_{k}) \left(|u_{A_1}|+\sum_{m=1}^{k-1}|u_{A_{m}}-u_{A_{m+1}} |\right)\\
\leq &C \sum_{k=1}^\infty \sum_{m=1}^{k-1} \mu_{\omega}(A_{k}) 2^{-m}R \jint_{A_{m}\cup A_{m+1}} g_{u,\rho}d\mu_{\omega}
+C\, R\, \left(\jint_{A_1}g_{u,\rho}^2\, d\mu_\omega\right)^{1/2}\sum_{k=2}^\infty \mu_\omega(A_k)\\
\leq &C \sum_{k=1}^\infty \sum_{m=1}^{k-1} \mu_{\omega}(A_{k}) 2^{-m}R \jint_{A_{m}\cup A_{m+1}} g_{u,\rho}d\mu_{\omega}
+C\, R\, \left(\jint_{A_1}g_{u,\rho}^2\, d\mu_\omega\right)^{1/2}\mu_\omega(B_\rho\setminus A_1)\\
\leq & C \sum_{m=1}^\infty\sum_{k=m+1}^\infty 2^{-m}R
    \int_{A_{m}\cup A_{m+1}} g_{u,\rho}d\mu_{\omega} \frac{\mu_{\omega}(A_{k})}{\mu_{\omega}(A_{m})}
+C\, R\, \left(\jint_{A_1}g_{u,\rho}^2\, d\mu_\omega\right)^{1/2}\mu_\omega(B_\rho)\\
\leq & C\, \sum_{m=1}^\infty\sum_{k=m+1}^\infty 2^{-m}R \int_{A_{m}\cup A_{m+1}} g_{u,\rho}d\mu_{\omega}
+C\, R\, \left(\jint_{A_1}g_{u,\rho}^2\, d\mu_\omega\right)^{1/2}\mu_\omega(B_\rho)\\
\leq & C R \int_{B_{\rho}}g_{u,\rho}d\mu_{\omega} +C\, R\, \left(\jint_{A_1}g_{u,\rho}^2\, d\mu_\omega\right)^{1/2}\mu_\omega(B_\rho).
\end{split}
\]
Hence by H\"older's inequality and~\eqref{eq:A1-est},
\[
\int_{B_\rho}|u|\, d\mu_\omega\le C\, R\, \left(\jint_{B_\rho}g_{u,\rho}^2\, d\mu_\omega\right)^{1/2}\, \mu_\omega(B_\rho),
\]
which yields the $2$-Poincar\'e inequality.
\end{proof}

\subsection{John domain property}

In this section we will show that $\widehat{X}$ is a John domain. In fact, it is also a uniform domain (see for
example~\cite{GS}), but we do not need this stronger geometric condition in our paper.

\begin{prop}
The space $(\widehat{X}, d_\rho,\mu_\omega)$ is a John domain with John center $\infty$
and boundary $Z\times\{0\}$.
\end{prop}

\begin{proof}
Since for $0<y\le 1$ and $x\in Z$ we have that $\dist_{d_\rho}((x,y),Z\times\{0\})=\dist_d((x,y), Z\times\{0\})=y$, it
follows that the boundary of $\widehat{X}$ contains $Z\times\{0\}$. Moreover, if $x_1,x_2\in Z$ such that $d_X(x_1,x_2)\le 1/2$,
then $d_\rho((x_1,0),(x_2,0))= d((x_1,0),(x_2,0))=d_X(x_1,x_2)$. We now show that there is no other boundary point for
$\widehat{X}$. If $((x_k,y_k))_k$ is a Cauchy sequence in $\widehat{X}$ that does not converge, then we must
necessarily have that there is some $\tau>0$ such that $y_k\le \tau$ for each $k$. This follows from~\eqref{e:dist-infty}.
That is, the sequence lies in $Z\times(0,\tau]$. From the fact that $Z$ is compact, it follows that we must have
$y_k\to 0$, that is, the sequence converges to a point in $Z\times\{0\}$. Hence $\partial\widehat{X}=Z\times\{0\}$.

Let $(x,y)\in X$; then $y>0$. Let $\gamma:[0,\infty)\to X$ be given by $\gamma(t)=(x,y+t)$. Note that $\gamma$ is arc-length
parametrized with respect to the metric $d$, but not with respect to the metric $d_\rho$. For $t>0$, we see that
\[
\ell_\rho(\gamma\vert_{[0,t]})=\int_0^t \rho(y+s)\, ds
  =\begin{cases} \frac{1}{\beta-1}\left[\frac{1}{y^{\beta-1}}-\frac{1}{(y+t)^{\beta-1}}\right] &\text{ if }y\ge 1,\\
     1-y+\frac{1}{\beta-1}\left[1-\frac{1}{(t+y)^{\beta-1}}\right]&\text{ if }y<1\le y+t,\\
     t &\text{ if }y+t<1. \end{cases}
\]
Moreover, for $t>0$ we have that
\[
\dist_{d_\rho}(\gamma(t),Z\times\{0\})=\int_0^{y+t}\rho(s)\, ds
  =\begin{cases} 1+\frac{1}{\beta-1}\left[1-\frac{1}{(y+t)^{\beta-1}}\right]&\text{ if } y+t\ge 1,\\
    y+t&\text{ if }y+t\le 1. \end{cases}
\]
As $y>0$, it follows from the above two computations that for each $t>0$ we have
$\dist_{d_\rho}(\gamma(t), Z\times\{0\})\ge \ell_\rho(\gamma\vert_{[0,t]})$. Hence, $\widehat{X}$ is a John domain
with John center $\infty$ and John constant $C_J=1$.
\end{proof}

\subsection{Conclusion: the reconciliation}\label{SubSec-reconciliation}

Now we have the tools necessary to compare the two fractional Laplace operators in the case that
$(Z,d_Z,\mu_Z)$ is a compact doubling metric measure space supporting a $2$-Poincar\'e inequality: the first considered in~\cite{EbGKSS} and
corresponding to the operator $\mathcal{A}$ described in Subsection~\ref{subsect:p=2Operator}, and the second
constructed in this paper and denoted $(-\Delta_2)^\theta$ .

The domain $X=Z\times(0,\infty)$ was
considered in~\cite{EbGKSS}. Moreover, for $-1<a<1$ (corresponding to the relationship $\theta=(1-a)/2$),
$X$ is equipped with the weighted product measure $d\mu_X(x,y)=y^a\, d\mu_Z(x)\, dy$.
With such a measure, it is clear to see that the measure $\nu=\mu_Z$ satisfies the co-dimension condition~\eqref{eq:Co-Dim-intro}
with $\Theta=a+1$.

It was shown in~\cite{EbGKSS} that a function $u\in B^\theta_{2,2}(Z)$ satisfies
$\mathcal{A} u=f$ on $Z=\partial X$ if and only if its Cheeger-harmonic extension, also denoted $u$, to $X$ satisfies
$\int_X g_u^2\, d\mu_X<\infty$ \emph{and in addition}, satisfied
\begin{equation}\label{eq:lim-to-bdy}
\lim_{y\to 0^+} y^a\partial_y u=f.
\end{equation}
Here $\mathcal{A}$ on $Z$ was defined via the spectral decomposition theorem, corresponding
to the fractional Laplacian $(\Delta_2)^\theta$ in the Euclidean setting of~\cite{CS}.
Observe that as $(Z,\mu_Z)$ is compact, doubling, and supports a $2$-Poincar\'e inequality, the measure $\mu_X$ is doubling
and supports a $2$-Poincar\'e inequality as well; this was shown in~\cite{EbGKSS}.

Now, when we transform $X=Z\times[0,\infty)$ into $(\widehat{Z}, d_\rho, \mu_\omega)$ as described at the beginning
of Section~\ref{Sec:Reconcile}, we obtain a doubling metric measure space supporting a $2$-Poincar\'e inequality and
the co-dimension condition as outlined in Conditions~(H0), (H1), and~(H2) in the current paper. Furthermore,
with the transformation of the Cheeger differential structure as explained in this section, we also see that
functions on $X$ that were Cheeger $2$-harmonic in $Z\times(0,\infty)$ are also Cheeger $2$-harmonic in
$\Om:=\widehat{X}\setminus (Z\times\{0\})$; moreover, $d_\rho$ is locally isometric to the original metric $d$ near $Z=\partial\Om$,
and the measure $\mu_\omega=\mu_X$ on $Z\times[0,1)$. It follows that the condition $\lim_{y\to 0^+} y^a\partial_y u=f$
is satisfied by $u$ if and only if $u$ satisfies Condition~(c) of our Theorem~\ref{thm:equiv-intro}. Moreover,
the trace of a function on $\Om$ to $\partial\Om$ is the same trace as the one for functions on $Z\times(0,\infty)$ to $Z$.

Thus the correspondence between the construction of~\cite{EbGKSS} and ours is as follows.
\begin{itemize}
\item A function $u$ on $Z$
constructed in~\cite{EbGKSS} to solve the equation $\mathcal{A}u=f$, is extended as a Cheeger $2$-harmonic
function, denoted $\widehat{u}$, to $Z\times(0,\infty)$ so that the trace of $\widehat{u}$
at $\partial (Z\times(0,\infty)=Z$ is $u$. It is shown there that $\widehat{u}$ also satisfies~\eqref{eq:lim-to-bdy}.
\item It was also shown in~\cite{EbGKSS} that $\widehat{u}$ has finite Dirichlet energy in $Z\times(0,\infty)$.
\item The same function $\widehat{u}$ is then Cheeger $2$-harmonic in the transformed
domain $(Z\times(0,\infty), d_\rho, \mu_\omega)$,
as discussed above. Moreover, $\widehat{u}$ has finite Dirichlet energy in this transformed domain.
\item By Lemma~\ref{lem:infty-point} we know that the point $\infty$ is removable for Cheeger $2$-harmonicity
of functions with finite Dirichlet energy. Hence $\widehat{u}$ is Cheeger $2$-harmonic in
$\Om=\widehat{X}\setminus(Z\times\{0\})$.
\item Combining the above-listed points, we see that $\widehat{u}$ satisfies Condition~(c) of
Theorem~\ref{thm:equiv-intro}, and hence satisfies the equation $(-\Delta_2)^\theta u=f$ on $Z$, where
$\theta$ and $a$ are related by $\theta=(1-a)/2=1-\Theta/2$ (with $\Theta=a+1$).
Here, $(-\Delta_2)^\theta$ is as constructed using the bilinear form $\mathcal{E}_T$ as described
in Section~\ref{sec:construct-fractLap} corresponding to $p=2$.
\item Since solutions to both problems exist and are unique up to additive constants, it follows that
the two approaches give the same solution.
\end{itemize}

\section{APPENDIX: Removing boundedness condition on $f$ in~\cite{MS}}\label{Sec:Append}

In this appendix we gather together results, and their proofs, that are adaptations of the results from~\cite{MS} to our setting.
In particular, we replace the boundedness condition on $f$, as required in~\cite{MS}, with the more natural condition
$f\in L^{p'}(\partial\Om)$, where $p'=p/(p-1)$ is the H\"older dual of $p$.

\begin{lem}\label{lem:control-via-data}
Let $u$ be the solution to the Neumann boundary value problem on $\Om$ with boundary data $f\in L^{p'}(\partial\Om)$
such that $\int_\Om u\, d\mu=0$. Then
\[
\int_\Om|u|^p\, d\mu\le C\int_\Om |\nabla u|^p\, d\mu\le C^2\int_{\partial\Om}|f|^{p'}\, d\nu.
\]
The constant $C$ depends only on the structural constants of $\Om$.
\end{lem}

The above lemma is a consequence of~\cite[Proposition~4.1]{MS} when the Cheeger differential structure is replaced with the
upper gradients in the formulation of the Neumann boundary value problem.

\begin{proof}
Since $\Om$ is bounded and supports a $p$-Poincar\'e inequality, and since $\int_\Om u\, d\mu=0$, we have that
\begin{equation}\label{eq:PI-zeroMean}
\int_\Om|u|^p\, d\mu\le C\int_\Om|\nabla u|^p\, d\mu.
\end{equation}
Next, for each $v\in N^{1,p}(\Om)$ we let $I(v)=\int_\Om|\nabla v|^p\, d\mu-p\int_{\partial\Om}v\, f\, d\nu$. Since $u$ is
a minimizer of $I$ as shown in Theorem~\ref{thm:equiv-intro}, and as $I(0)=0$, it follows that $I(u)\le 0$, that is,
\[
\int_\Om|\nabla u|^p\, d\mu\le p\int_{\partial\Om}u\, f\, d\nu.
\]
Hence by the trace theorem~\cite[Theorem~1.1]{M}, we have
\begin{align*}
\int_\Om|\nabla u|^p\, d\mu\le p\, \Vert u\Vert_{L^p(\partial\Om)}\, \Vert f\Vert_{L^{p'}(\partial\Om)}
   &\le C\left(\Vert u\Vert_{L^p(\Om)}+\Vert |\nabla u|\Vert_{L^p(\Om)}\right)\, \Vert f\Vert_{L^{p'}(\partial\Om)}\\
   &\le C \Vert |\nabla u|\Vert_{L^p(\Om)}\, \Vert f\Vert_{L^{p'}(\partial\Om)},
\end{align*}
where we used~\eqref{eq:PI-zeroMean} in the last line above. It follows that
\[
\Vert |\nabla u|\Vert_{L^p(\Om)}^{p-1}\le C\, \Vert f\Vert_{L^{p'}(\partial\Om)},
\]
from whence the second inequality in the statement of the lemma follows.
\end{proof}

The proof of boundedness of solutions to the Neumann problem as given in~\cite{MS} uses the boundedness of the
Neumann data $f$. In this appendix we show how to modify the proof of~\cite[Theorem~5.2]{MS} when relaxing
the requirement that $f\in L^\infty(\partial\Om)$ to $f\in L^q(\partial\Om)$ for sufficiently large $q>p$. We also
relax the co-dimension $1$ condition to a co-dimension $\Theta$ condition, namely that there is some $0<\Theta<p$
so that for each $x\in\partial\Omega$ and $0<r\le \diam(\partial\Omega)$ we have
\begin{equation}\label{eq:codim-Theta}
\nu(B(x,r))\approx \frac{\mu(B(x,r))}{r^\Theta}.
\end{equation}
Here 
the ball $B(x,r)$ is the ball in $\overline{\Omega}$, with
$\mu(\partial\Omega)=0$ and $\nu(\Omega)=0$. Observe that since $\mu$ is a doubling measure, there is an exponent
$s>0$ such that for all $x\in\overline{\Omega}$ and $0<r<R\le \diam(\Omega)$,
\begin{equation}\label{eq:mu-lower-bdd-exponent}
\left(\frac{r}{R}\right)^s\lesssim\frac{\mu(B(x,r))}{\mu(B(x,R))}.
\end{equation}
Note that we can make $s$ as large as we like; thus, once $\Theta<p$ is fixed, with $1<p<\infty$, we can then  choose
$s>0$ so that $p<s$. Moreover, given the assumptions on $\Omega$, we can replace the original metric with
the biLipschitz equivalent length-metric, and so we can assume that $\Omega$ is a length space (the class
of upper gradients and the Cheeger differential structure are not changed by this); it follows from~\cite[Proposition~4.21]{M}
that given  a choice of $p<\widetilde{p}<p^*:=\tfrac{p(s-\Theta)}{s-p}$, the trace operator
$T:N^{1,p}(\Omega)\to L^{\widetilde{p}}(\partial\Omega)$ is local in the sense that whenever $x\in\partial\Omega$
and $r>0$, we have
\begin{equation}\label{eq:Mal-local-trace}
\Vert Tu-u_{B(x,r)}\Vert_{L^{\widetilde{p}}(B(x,r)\cap\partial\Omega)}
  \lesssim r^{({\widetilde{p}}^{-1}-{p^*}^{-1})(s-\Theta)}\, \Vert \nabla u\Vert_{L^p(B(x,r))}.
\end{equation}
Up to Remark~5.9 of~\cite{MS} holds without any change even if we only assume that $f\in L^{p'}(\partial\Omega)$
where $p'=p/(p-1)$ is the H\"older dual of $p$. The subsequent parts of Section~5 of~\cite{MS} are modified as follows.

Lemma~5.10 of~\cite{MS} is modified to the following. Note that if $f$ is bounded, then the integrability condition of $f$
given in the following lemma holds for any choice of $q$ satisfying the conditions set out in the lemma.
In the lemma below, $\widetilde{p}\in (p,p^*)$ is the exponent associated with the trace operator as in~\eqref{eq:Mal-local-trace},
and $\widetilde{p}^\prime=\widetilde{p}/(\widetilde{p}-1)$ is the H\"older dual of $\widetilde{p}$ and $p'$ the H\"older dual of $p$.

For ease of notation, in the following, integrals over balls $B$ with respect to $\mu$ stand in for
integrals over $B\cap\Om$, and integrals with respect to $\nu$ stand in for integrals over $B\cap\partial\Om$.

\begin{lem}
Suppose that $\int_{\partial\Omega}f\, d\nu=0$ with $f\in L^{p'}(\partial\Omega)$,
and set $t=\tfrac{p'-\widetilde{p}'}{p'-1}$.
Let $u$ be a solution to the Neumann boundary
value problem with boundary data $f$. For $x\in\partial\Omega$, $k,h\in\R$ with $k>h$,
and $0<R/2\le r<R\le \diam(\partial\Omega)$,
set
\[
u(k,r):=\left(\jint_{B(x,r)}(u-k)_+^p\, d\mu\right)^{1/p},\qquad
 \psi(k,r):=\jint_{B(x,r)}|f|(u-k)_+\, d\nu.
\]
Then
\begin{align}\label{eq:level-modified}
u(k,r)
&\le C\left(\frac{u(h,R)}{k-h}\right)^{1-1/\kappa}\left[ \frac{R}{R-r} u(h,R)+R^{1-\Theta/p}\, \psi(h,R)^{1/p}\right],\notag\\
\psi(k,r)
&\le C_{f,R}\left(\frac{\psi(h,R)}{k-h}\right)^{t/\widetilde{p}^\prime}
 \left[\frac{R^{\aleph_1}}{R-r}u(h,R)+R^{\aleph_2}\, \psi(h,R)^{1/p}\right],
\end{align}
where $\kappa=s/(s-p)$ and the constants $\aleph_1,\aleph_2$ are positive and independent of $f$, $u$, $x$, $r$, and $R$,
while $C_{f,R}>1$ is independent of $u$, $x$, and $r$.
\end{lem}

The choice of $\kappa=s/(s-p)$ comes from the Sobolev-Poincar\'e inequality on $\overline{\Omega}$,
namely, $(\kappa p,p)$-Poincar\'e inequality; see~\cite{HajK}. We remind the readers that the parameter $s$ is the
lower mass bound dimension of $\mu$ from~\eqref{eq:mu-lower-bdd-exponent}.

\begin{proof}
As in the proof of~\cite[Lemma~5.10]{MS}, we pick a cut-off function $\eta$ that is Lipschitz on $\overline{\Omega}$,
with $\eta=1$ on $B(x,r)$, $\eta=0$ on $\overline{\Omega}\setminus B(x,(R+r)/2)$ and apply
H\"older's inequality and the $(\kappa p,p)$-Poincar\'e
inequality to $\eta \, (u-k)_+$ to obtain
\[
\jint_{B(x,r)}(u-k)_+^p\, d\mu\lesssim \left(\frac{\mu(A(k,r))}{\mu(B(x,r))}\right)^{1-1/\kappa}\, R^p
\jint_{B(x,(R+r)/2)}|\nabla (\eta (u-k)_+)|^p\, d\mu.
\]
Here $A(k,r)=\{y\in\overline{\Omega}\, :\, u(y)>k\}\cap B(x,r)$.
Invoking the Leibniz rule (with $|\nabla \eta|\le 2(R-r)^{-1}\chi_{B(x,(R+r)/2)\setminus B(x,r)}$) and
the version of De Giorgi-type inequality from~\cite[Theorem~5.3]{MS}, we now obtain
\[
 u(k,r)^p\lesssim  \left(\frac{\mu(A(k,r))}{\mu(B(x,r))}\right)^{1-1/\kappa}\,
\left[\frac{R^p}{(R-r)^p}\jint_{B(x,R)}(u-k)_+^p\, d\mu+\frac{1}{\mu(B(x,R))}\int_{B(x,R)}|f|(u-k)_+\, d\nu\right].
\]
Now applying the codimensionality~\eqref{eq:codim-Theta}, we obtain
\[
u(k,r)^p\lesssim  \left(\frac{\mu(A(k,r))}{\mu(B(x,r))}\right)^{1-1/\kappa}\,
\left[\frac{R^p}{(R-r)^p}u(k,R)^p+R^{p-\Theta}\jint_{B(x,R)}|f|(u-k)_+\, d\nu\right].
\]
Now, as in~\cite[bottom of page~2446]{MS}, we  see that when $h<k$,
\[
\mu(A(k,r))\lesssim\frac{\mu(B(x,r))}{(k-h)^p}\, u(h,R)^p.
\]
Combining the above two inequalities proves the first of the two inequalities in~\eqref{eq:level-modified} (note that
we have used the doubling property of $\mu$ as well as the facts that $R/2\le r<R$ and $u(k,R)\le u(h,R)$,
$\psi(k,R)\le \psi(h,R)$ here.

Now we turn our attention to $\psi(k,r)$ where the bulk of the modification lies.
We set $\aleph_0:=(\widetilde{p}^{-1}-{p^*}^{-1})(s-\Theta)$.
By H\"older's inequality and~\eqref{eq:Mal-local-trace}, H\"older's inequality again, and then by the definition of $u(k,r)$,
\begin{align*}
\psi(k,r)&\le \left(\jint_{B(x,r)}(u-k)_+^{\widetilde{p}}\, d\nu\right)^{1/\widetilde{p}}
   \left(\jint_{B(x,r)}|f\chi_{A(k,r)}|^{\widetilde{p}'}\, d\nu\right)^{1/\widetilde{p}'}\\
&\lesssim\nu(B(x,r))^{-1/\widetilde{p}}\left(\jint_{B(x,r)}|f\chi_{A(k,r)}|^{\widetilde{p}'}\, d\nu\right)^{1/\widetilde{p}'}\, \times\\
 &\hskip 1cm  \left[r^{\aleph_0}\left(\int_{B(x,r)}|\nabla (u-k)_+|^p\, d\mu\right)^{1/p}
  +\nu(B(x,r))^{1/\widetilde{p}}\jint_{B(x,r)}(u-k)_+\, d\mu\right]\\
&\lesssim \left(\jint_{B(x,r)}|f\chi_{A(k,r)}|^{\widetilde{p}'}\, d\nu\right)^{1/\widetilde{p}'}
\left[\frac{r^{\aleph_0}}{\nu(B(x,r))^{1/\widetilde{p}}}\left(\int_{B(x,r)}|\nabla (u-k)_+|^p\, d\mu\right)^{1/p}
   +u(k,r)\right].
\end{align*}
Combining this inequality with the De Giorgi-type inequality from~\cite[Theorem~5.3]{MS} yields
\begin{align*}
\psi(k,r)\lesssim  \left(\jint_{B(x,r)}|f\chi_{A(k,r)}|^{\widetilde{p}'}\, d\nu\right)^{1/\widetilde{p}'}
  \bigg[u(k,r)+&\frac{r^{\aleph_0}}{R-r}\frac{\mu(B(x,R))^{1/p}}{\nu(B(x,r))^{1/\widetilde{p}}}u(k,R)\\
&  +r^{\aleph_0}\nu(B(x,R))^{1/p-1/\widetilde{p}}\psi(k,R)^{1/p}\bigg].
\end{align*}
Here we have also used the facts that $nu$ is doubling and $R/2\le r<R$. Now using the co-dimensionality of $\nu$ with
respect to $\mu$, we obtain
\begin{align*}
\psi(k,r)\lesssim  \left(\jint_{B(x,r)}|f\chi_{A(k,r)}|^{\widetilde{p}'}\, d\nu\right)^{1/\widetilde{p}'}
  \bigg[u(k,r)+&\frac{R^{\aleph_0+\Theta/p}}{R-r}\nu(B(x,R))^{1/p-1/\widetilde{p}}u(k,R)\\
    &+r^{\aleph_0}\nu(B(x,R))^{1/p-1/\widetilde{p}}\psi(k,R)^{1/p}\bigg].
\end{align*}
As $\widetilde{p}>p$ and $\nu(\partial\Omega)<\infty$, we have
\begin{equation}\label{ineq:psi}
\psi(k,r)\lesssim \left(\jint_{B(x,r)}|f\chi_{A(k,r)}|^{\widetilde{p}'}\, d\nu\right)^{1/\widetilde{p}'}
   \left[\left(1+\frac{R^{\aleph_0+\Theta/p}}{R-r}\right)u(k,R)+r^{\aleph_0}\psi(k,R)^{1/p}\right].
\end{equation}
Note that by assumption, $f\in L^{p'}(\partial\Omega)$. As $\widetilde{p}>p$, it follows that
$f\in L^{\widetilde{p}'}(\partial\Omega)$.
We set $t=\tfrac{p'-\widetilde{p}'}{p'-1}$ as in the statement of the lemma. 
As $p<\widetilde{p}$ and $p'>1$, we have that $t>0$.
Moreover, as $\widetilde{p}'>1$, it follows that $t<1$. Now by H\"older's inequality
and the fact that $(\widetilde{p}'-t)/(1-t)=p'$,
\begin{align*}
\jint_{B(x,r)}|f\chi_{A(k,r)}|^{\widetilde{p}'}\, d\nu &\le \left(\jint_{B(x,r)}|f|\chi_{A(k,r)}\, d\nu\right)^t\,
  \left(\jint_{B(x,r)}|f|^{(\widetilde{p}'-t)/(1-t)}\, d\nu\right)^{1-t}\\
    &=\left(\jint_{B(x,r)}|f|\chi_{A(k,r)}\, d\nu\right)^t\, \left(\jint_{B(x,r)}|f|^{p'}\, d\nu\right)^{(\widetilde{p}'-1)/(p'-1)}\\
    &\le C\, \left(\jint_{B(x,r)}|f|\chi_{A(k,r)}\, d\nu\right)^t\, \left(\jint_{B(x,R)}|f|^{p'}\, d\nu\right)^{(\widetilde{p}'-1)/(p'-1)}.
\end{align*}
We set
\[
C_{f,R}:=C\, \left(\jint_{B(x,R)}|f|^{p'}\, d\nu\right)^{(\widetilde{p}'-1)/(p'-1)}
\]
Then by the above,
\[
\jint_{B(x,r)}|f\chi_{A(k,r)}|^{\widetilde{p}'}\, d\nu\le C_{f,R}\left(\jint_{B(x,r)}|f|\chi_{A(k,r)}\, d\nu\right)^t.
\]

When $h,k\in\R$ with $h<k$, we have
\[
\jint_{B(x,r)}|f|\chi_{A(k,r)}\, d\nu\le \frac{1}{k-h}\jint_{B(x,r)}|f|(u-h)_+\, d\nu=\frac{\psi(h,r)}{k-h}
   \le \frac{\psi(h,R)}{k-h}.
\]
Combining this with inequality~\eqref{ineq:psi}, we obtain
\[
\psi(k,r)\lesssim C_{f,R}^{1/\widetilde{p}'}\left(\frac{\psi(h,R)}{k-h}\right)^{t/\widetilde{p}'}  
  \left[\left(1+\frac{R^{\aleph_0+\Theta/p}}{R-r}\right)u(k,R)+R^{\aleph_0}\psi(k,R)^{1/p}\right].
\]
As $R\le \diam(\partial\Omega)<\infty$ and $R>R-r$, we obtain the second part of~\eqref{eq:level-modified}
by choosing $\aleph_2=\aleph_0$ and $\aleph_1=\aleph_0+\Theta/p$.
\end{proof}

Now we are ready to prove the boundedness of $u$ on $B(x,R)$ even when $f$ is not bounded. To do so, fix $k_0\in\R$.
We fix $d>0$ for now, but we will add some conditions in on $d$ towards the end.
For non-negative integers $n$ we set
\[
r_n=\frac{R}{2}(1+2^{-n}),\qquad k_n=k_0+d(1-2^{-n}).
\]
Note that $0<k_{n+1}-k_n=2^{-(n+1)}d$ and $0<r_n-r_{n+1}=2^{-(n+2)}R$. Then from the above lemma, we see
that
\[
u(k_{n+1},r_{n+1})\le C\left(\frac{u(k_n,r_n)}{2^{-(n+1)}d}\right)^{1-1/\kappa}
  \left[2^{n+2}u(k_n,r_n)+R^{1-\Theta/p}\psi(k_n,r_n)^{1/p}\right].
\]
Setting $\alpha=1-1/\kappa<1$ and noting that $R\le \diam(\partial\Omega)<\infty$, we obtain
\[
u(k_{n+1},r_{n+1})\le C\frac{2^{(n+1)\alpha}}{d^\alpha}\left[2^{n+2}u(k_n,r_n)^{1+\alpha}
   +u(k_n,r_n)^\alpha\psi(k_n,r_n)^{1/p}\right],
\]
and as $2^{n+2}\ge 1$, we obtain
\begin{equation}\label{eq:star1}
u(k_{n+1},r_{n+1})\le C\frac{2^{n(\alpha+1)}}{d^\alpha}
   \left[u(k_n,r_n)^{1+\alpha}+u(k_n,r_n)^\alpha\psi(k_n,r_n)^{1/p}\right].
\end{equation}
Similarly, setting $\beta=t/\widetilde{p}'<1$, we have from the above lemma that
\begin{equation}\label{eq:star2}
\psi(k_{n+1},r_{n+1})\le C_{f,R}\frac{2^{n(\beta+1)}}{d^\beta}
  \left[u(k_n,r_n)\psi(k_n,r_n)^\beta+\psi(k_n,r_n)^{\beta+1/p}\right].
\end{equation}

\begin{lem}\label{lem:bdd-MalySh}
There exist a choice of positive real numbers $d$, $\sigma$, and $\tau$ such that
with $k_n$, $r_n$ as defined above for each non-negative integer $n$, we have
\[
u(k_n,r_n)\le 2^{-\sigma n}u(k_0,R),\qquad \psi(k_n,r_n)\le 2^{-\tau n}\psi(k_0,R).
\]
\end{lem}

Note that with $d,\sigma$, and $\tau$ as in the above lemma, we obtain that
\[
0=\lim_{n\to\infty}u(k_n,r_n)=u(k_0+d,R/2),
\]
form whence we can conclude that $u\le k_0+d$ on $B(x,R/2)$.
Hence, to prove the boundedness of $u$ on $B(x,R/2)$ it suffices to prove the above lemma.

\begin{proof}
If $u(k_0,R)=0$, then $u\le k_0$ on $B(x,R)$, and we have the boundedness. Therefore, without loss of generality, we
assume that $u(k_0,R)>0$. If $\psi(k_0,R)=0$, then again for each $n$ we have $\psi(k_n,r_n)=0$, and the
required inequality for $\psi$ would be satisfied, and then we can directly focus on the part of the
proof below that is relevant to $u(k_n,r_n)$. Hence we now assume also that $\psi(k_0,R)>0$, and prove the claim
via induction on $n$. The base case $n=0$ holds trivially. So suppose that $n$ is a non-negative integer such that the
claims hold for $n$; and we wish to then show that the claims also hold for $n+1$.

By~\eqref{eq:star2} and by the hypotheses holding for $n$,
\begin{align*}
\psi(k_{n+1},r_{n+1})&\le C\frac{2^{n(\beta+1)}}{d^\beta}\left[2^{-(\tau\beta+\sigma)n}u(k_0,R)\psi(k_0,R)^\beta
 +2^{-\tau n(\beta+1/p)}\psi(k_0,R)^{\beta+1/p}\right]\\
 &\le \frac{C}{d^\beta}\left[2^{-n(\sigma+\tau\beta-\beta-1)}\psi(k_0,R)^\beta u(k_0,R)
  +2^{-n(\tau\beta+\tau/p-\beta-1)}\psi(k_0,R)^{\beta+1/p}\right]\\
&\le 2^{-\tau(n+1)}\psi(k_0,R)\, \frac{2^\tau C}{d^\beta}
   \bigg[2^{-n(\sigma+\tau\beta-\beta-1-\tau)}\frac{u(k_0,R)}{\psi(k_0,R)^{1-\beta}}\\
& \hskip 5cm + 2^{-n(\tau\beta+\tau/p-\beta-1-\tau)}\psi(k_0,R)^{\beta+1/p-1}\bigg].
\end{align*}
Similarly, by~\eqref{eq:star1} and by the hypothesis,
\begin{align*}
u(k_{n+1},r_{n+1})\le 2^{-\sigma(n+1)}u(k_0,R)\frac{C}{d^\beta}
 \bigg[2^{-n(\sigma+\sigma\alpha-\alpha-1-\sigma)}&u(k_0,R)^\alpha\\
   &+2^{-n(\sigma\alpha+\tau/p-\alpha-1-\sigma)}\frac{\psi(k_0,R)^{1/p}}{u(k_0,R)^{1-\alpha}}\bigg].
\end{align*}
We choose $\sigma$ and $\tau$ such that
\begin{align*}
\sigma+\sigma\alpha-\alpha-1-\sigma\ge0,\qquad & \sigma\alpha+\frac{\tau}{p}-\alpha-1-\sigma\ge0, \\
\sigma+\tau\beta-\beta-1-\tau\ge 0, \qquad & \tau\beta+\frac{\tau}{p}-\beta-1-\tau\ge 0.
\end{align*}
Such a choice is possible because $\beta=t/\widetilde{p}'<1<p'=p/(p-1)$, see the conditions
listed at the top of~\cite[page~2450]{MS}.
Subsequently, we choose $d$ large enough so that
\[
d^\beta\ge C\left[u(k_0,R)^\alpha+\frac{u(k_0,R)}{\psi(k_0,R)^{1-\beta}}+\psi(k_0,R)^{\beta+1/p-1}
  \frac{\psi(k_0,R)^{1/p}}{u(k_0,R)^{1-\alpha}}\right]
\]
to complete the proof of the lemma.
\end{proof}

\noindent {\bf Addresses:} \\

\noindent L.C.: ~Department of Mathematics and Statistics, Smith College, Northampton, MA~01063, U.S.A. \
E-mail: {\tt lcapogna@smith.edu}\\
\\
\noindent J.K.:~Department of Mathematical Sciences, P.O.~Box 210025, University of Cincinnati, Cincinnati, OH~45221-0025, U.S.A.\
E-mail: {\tt klinejp@mail.uc.edu}\\
\\
\noindent R.K.:~Aalto University, Department of Mathematics and Systems Analysis, P.O.~Box~11100, FI-00076 Aalto, Finland.\
E-mail: {\tt riikka.korte@aalto.fi}\\
\\
\noindent N.S.:~Department of Mathematical Sciences, P.O.~Box 210025, University of Cincinnati, Cincinnati, OH~45221-0025, U.S.A.\
E-mail: {\tt shanmun@uc.edu}\\
\\
\noindent M.S.:~Department of Mathematics and Statistics, Kenyon College, 201 N. College Dr., Gambier, OH~43022, U.S.A.\
E-mail: {\tt snipesm@kenyon.edu}

\end{document}